\documentclass[11pt,reqno]{amsart}

\usepackage[
a4paper,
twoside, hmargin={22mm,25mm}, vscale=0.78, vmarginratio=4:5, headheight=15pt, bindingoffset=15mm]{geometry}

\usepackage{amsthm}
\usepackage{amsmath}
\usepackage{eucal}

\usepackage{amssymb}
\usepackage{dsfont} 

\usepackage{ifthen}
\usepackage{mathrsfs}
\usepackage{enumerate}

\newcommand{\Scr}[2]{{\left(#1\,,\:#2\right)}}
\newcommand{\Sce}[2]{{\left\langle #1\,,\:#2\right\rangle}}
\newcommand{\scr}[2]{{(#1\,,\:#2)}}
\newcommand{\sce}[2]{{\langle #1\,,\:#2\rangle}}

\newcommand{\norm}[1]{{\|#1\|}}

\newcommand{\FF}{\mathrm{F}}
\newcommand{\Reell}{\mathds{R}}
\newcommand{\Complex}{\mathds{C}}
\newcommand{\Rational}{\mathds{Q}}
\newcommand{\Ganz}{\mathds{Z}}
\newcommand{\Nat}{\mathds{N}}

\newcommand{\defgleich}{\mathrel{\mathop:}=} 

\newcounter{szaehler}
\newcounter{tzaehler}
\numberwithin{szaehler}{section}
\theoremstyle{plain} 
\newtheorem{thm}[szaehler]{Theorem}
\newtheorem{xthm}[tzaehler]{Theorem}
\newtheorem{satz}[szaehler]{Proposition}

\newtheorem{cor}[szaehler]{Corollary}

\newtheorem{lem}[szaehler]{Lemma}

\theoremstyle{definition}
\newtheorem{defn}[szaehler]{Definition}

\newenvironment{bem}{\textbf{Remark \stepcounter{szaehler}\arabic{section}.\arabic{szaehler}.}\hspace{1em}}{~\hfill$\diamond$}

\newenvironment{introenv}[2]{\begin{center}\begin{minipage}{12cm}\textbf{#1 \ref{#2}.}\par\begin{it} }{\end{it}\end{minipage}\end{center}}

\newcommand{\eps}{\varepsilon}
\newcommand{\bei}[1]{\hspace{0.2ex}\rule[-1.5ex]{0.04em}{3ex}{\rule[-0.9ex]{0pt}{0pt}}_{\hspace{0.2ex}#1}}

\newcommand{\tensor}{\otimes}

\newcommand{\be}{\begin{equation}}
\newcommand{\ee}{\end{equation}}

\DeclareMathOperator{\supp}{supp}

\DeclareMathOperator{\vol}{vol}
\DeclareMathOperator{\sign}{sign}

\newcommand{\astm}{\ast^{\scriptscriptstyle M}}
\newcommand{\astz}{\ast^{\scriptscriptstyle Z}}
\newcommand{\ua}{\underline{a}}
\newcommand{\ud}{\underline{d}}
\newcommand{\tC}{{\widetilde C}}
\newcommand{\tE}{{\widetilde E}}

\renewcommand{\H}{{\mathscr H}}

\renewcommand{\Re}{\text{Re}\,}
\renewcommand{\Im}{\text{Im}\,}
\renewcommand{\rho}{\varrho}
\newcommand{\shriek}{{!}}
\DeclareMathOperator{\res}{res}
\DeclareMathOperator{\im}{im}
\DeclareMathOperator{\dom}{dom}
\DeclareMathOperator{\bild}{im}

\newcommand{\ostrich}[1]{#1\hspace{-1.5ex}\raisebox{1.9ex}{\rule{1.4ex}{0.3pt}}}
\newcommand{\bH}{{\ostrich{H}}}
\newcommand{\bX}{{\bar{X}}}

\newcommand{\pfrac}{\binom}

\allowdisplaybreaks[2]

\DeclareMathOperator{\FP}{FP}

\newcommand{\Hor}{{\mathcal H}}
\renewcommand{\Vert}{{\mathcal V}}

\title{A Hodge-type Theorem for Manifolds\\ with fibered cusp metrics}
\author{J\"orn M\"uller}

\begin{document}

\begin{abstract}
A manifold with fibered cusp metrics $X$ can be considered as a geometrical generalization of locally symmetric spaces of $\Rational-$rank one at infinity. We prove a Hodge-type theorem for this class of Riemannian manifolds, i.e.\ we find harmonic representatives of the de Rham cohomology $H^p(X)$. Similar to the situation of locally symmetric spaces, these representatives are computed by special values or residues of generalized eigenforms of the Hodge-Laplace-Operator on $\Omega^p(X)$.
\end{abstract}

\maketitle
\section*{Introduction and Statement of Results}

The classical theorems of geometric topology, such as the Hodge theorem, the signature
theorem and the index theorem, reveal a profound relationship between analysis of 
classical linear operators over a smooth compact manifold and the topology of the manifold.
Since their proofs in the middle of the last century, there has been much interest in 
extending these theorems to more general settings, where the manifolds may either
be noncompact or may have singularities.  The analytic approaches have often involved
creation of pseudodifferential operator calculi suited to certain geometric settings, 
such as the $b$-calculus of Melrose.  Another approach has sprung from techniques
in analytic number theory, and the analysis of the Laplacian over noncompact 
 locally symmetric spaces, as for example in the work of G.~Harder.

Recall that the classical Hodge theorem states that the natural map from harmonic
forms over a compact smooth manifold to de Rham cohomology classes is an isomorphism.
%
%
In the situation when  $X$ is not compact or is not smooth, there is no such general statement. One possible extension is to consider square integrable harmonic forms $\H_{(2)}(X)$ and identify them  with a topologically defined space.
For several geometric situations,  including manifolds with cylindrical ends (\cite{APS}), conical singularities (\cite{chgr}) and locally symmetric spaces (e.g. \cite{zucker}, \cite{saper}), such theorems of ``Hodge type'' have been found.

Manifolds with fibered cusp metrics can be considered as a geometrical generalization of $\Rational$-rank one locally symmetric spaces at ``infinity'' as well as of manifolds with cusps or cylindrical ends.
In \cite{hunsicker1} methods from the  $\phi-$calculus developed by Melrose \cite{Mel}, Mazzeo, Vaillant \cite{vaill} and others 
have been used
to find an identification of $\H_{(2)}^p(X)$ with a subspace of the intersection cohomology.

We want to take another approach and identify the de Rham cohomology of a manifold with fibered
cusps with a space of harmonic forms.
Generally these forms will not be square integrable.
In the early paper \cite{harder} of G.~Harder such a theorem for locally symmetric spaces of $\Rational$-rank one was proved. In this situation the representatives of the de Rham cohomology classes are either $L^2-$harmonic forms, or they are defined by special harmonic values of Eisenstein series. In \cite{mu-unpub} W.~M\"uller suggested to use  analytical arguments to find a similar theorem in the context of manifolds with cusps, by replacing  Eisenstein series with generalized eigenforms of the Laplacian. His method relies on
an explicit parametrix construction for the resolvent which allows
an investigation of the scattering- and spectral theory of the underlying manifold.
We take these ideas as the starting point to prove a Hodge-type theorem for a large class of manifolds with fibered cusp metrics.

For a Riemannian submersion $M\to B$ with $f$-dimensional standard fiber $F$ we equip
$Z=\mathbb{R}^+\times M$ with the Riemannian metric
$g^Z=du^2+\pi^*g^B+e^{-2u} g^{F_b}$, where $g^B$ and $g^{F_b}$ are the Riemannian metrics on
 $B$ resp.~the vertical tangent bundle $TF$.
A Riemannian manifold $X$ is called manifold with fibered cusp metric, if  $X$ is isometric to $(Z, g^Z)$ outside a compact set.
As long as the fibers are not points, $X$ is a complete manifold with finite volume. The splitting
$TM=\pi^*TB\oplus TF$ of the tangent bundle induces an isomorphism $\Omega^*(M)=\Omega^*(B,W)$, where
$W$ is the vector bundle whose fiber over $b\in B$ is $C^\infty(F_b, \Lambda(T^*F)|_{F_b})$. On $W$ one can define a Hermitean metric and thus the fiber-wise ``vertical''
Laplacian.  In this way we can 
define `fiber-harmonic forms' over $M$, and from there, over $Z$.   These turn out to play
an crucial role in the spectral theory of the Laplacian on $X$.  

For the analysis of the spectral theory it is necessary to ensure that fiber-harmonic forms are an invariant subspace of $\Delta_Z$. To that end we impose two obstructions on the submersion $M\to B$, namely (A) that the horizontal distribution is integrable and (B) that fiber-harmonic forms are an invariant subspace of the ``horizontal Laplacian'' $\Delta_{1,0}=d^{1,0}\delta^{1,0}+\delta^{1,0}d^{1,0}$.
To find a geometrical interpretation of condition (B) is an open question, but
we can prove the following sufficient criterion
\begin{introenv}{Proposition}{satzproj}
If the mean curvature $H$ of the fibers is projectable, then  the decomposition \eqref{faserorth} is invariant under $d^{1,0}$ and $\delta^{1,0}$.
\end{introenv}

Let $r:H^p(X)\to H^p(M)$ be the standard restriction map induced from the inclusion $M \subset X$. Using generalized eigenforms of $\Delta_X$ we will explicitely construct a map $\Xi$ from $\H^p(M)$ into the smooth harmonic forms on $X$, which extends to a map $H^p(M)\to H^p(X)$.
We obtain the following theorem:
\begin{introenv}{Theorem}{thm1}

Let $H_!^p(X)\defgleich\bild (H_c^p(X)\to H^p(X))$ be the image of cohomology with compact support in  the de~Rham-cohomology.
Let $H_{\mathrm{inf}}^p(X)$ be a complementary space to $H_!^p(X)$ in $H^p(X)$,\index{$H_{\rm inf}^p(X)$}\index{$H_{\shriek}^p(X)$}
\begin{equation}
H^p(X)=H_!^p(X)\oplus H_{\mathrm{inf}}^p(X).\label{thsplit}
\end{equation}
Let $R^p\defgleich\bild (r: H^p(X)\to H^p(M))$.
Then $\Xi(R^p)$ is isomorphic to $H_{\text{\upshape{inf}}}^p(X)$
and\/ $\Xi(H^p(M))=\Xi(R^p).$
\end{introenv}
Since all classes on the right hand side of \eqref{thsplit} have harmonic representatives, this indeed is a ``Hodge-type'' theorem.

The paper is organized as follows.  
First we investigate the spectral theory of the Hodge-Laplace operator on $X$. To that end we first examine the spectral theory of the noncompact end $Z$ using the Friedrichs extension of the Laplacian on compactly supported forms.
As mentioned above, fiber-harmonic forms play a key role. In the classical theory of automorphic forms, eigenforms of the Laplacian which are orthogonal to fiber-harmonic forms are known as ``cusp forms''.
It turns out that fiber-harmonic forms determine the essential spectrum, whereas  
the cusp forms are associated with the point spectrum of $\Delta_Z$, i.e. they are $L^2-$eigenforms (\textbf{Propositions  \ref{kptr}} and \textbf{\ref{wspek}}).

One important consequence of the two conditions $A$ and $B$ introduced above is \textbf{Proposition \ref{satzspec}}, which states that the de Rham cohomology of $M$ can be identified with $\Delta_{1,0}$-harmonic forms on $B$ with values in harmonic sections in the fibers:
\[
 H^p(M)\cong\bigoplus_{r+s=p} \H^r(B,\H^s(F)).
\]
This allows the parametrix construction of the resolvent for the Laplacian from the setting of manifolds with cusps (e.g. \cite{mu-rank1}) to be carried out in our more general fibered setting (section \ref{parakapform}).
With this explicit knowledge of the resolvent kernel, the spectral decomposition of $\Delta_Z$  can be computed.
Furthermore, an argument from mathematical scattering theory shows
\begin{introenv}{Proposition}{acspec1}
The absolutely continuous part $L^2_{ac}\Omega^p(X)$ of $\dom\Delta_X$  is unitarily equi\-valent to the fiber-harmonic forms $\Pi_0L^2\Omega^p(Z)$.
\end{introenv}
In this sense, the spectral theory on $X$ is determined by the spectral theory of $\Delta$ on $Z$.


The spectral resolution of $L^2_{ac}\Omega^p(X)$ is given by generalized eigenforms (GEs). In the theory of locally symmetric spaces, these are given by Eisenstein series. Since GEs are used frequently in this paper, it is useful to briefly recall some facts about them. Their construction in section \ref{keif} is similar to \cite{mu-cusprk1}.
For every fiber-harmonic form on $M$--or every cohomology class in $H^p(M)$--there exists a GE on $X$ which depends on a parameter on an infinite covering of $\Complex$, the spectral surface. Near $0$ this covering is twofold, and the  $E(s,\phi)$ are parametrized by harmonic forms $\phi$ on $M$. 
Let $\phi\in \H^{p-k}(B,\H^k(F))$ and set $d_k=|f/2-k|$. Then
 $s\mapsto E(s,\phi)$ is meromorphic for $s\in U\subset \Complex$ and $E(s,\phi)\in \Omega^p(X)$ satisfies a growth condition on the end $Z$. Furthermore
\[
\Delta E(s,\phi)=s(2d_k-s)E(s,\phi),
\]
and $E(s,\phi)$ is uniquely determined by these conditions.
From the conditions one can also conclude that the asymptotic expansion on the end $Z$ of the fiber-harmonic part of $E$ is given by
\begin{equation}
 \Pi_0 E(s,\phi) = e^{(f/2-k-d_k+s) r}\phi+\sum_{l=0}^f e^{(f/2-l+d_l-s) r}T^{[l]}(s)(\phi)
+G(s,\phi),\label{introasym}
\end{equation}
where $G(s,\phi)\in L^2\Omega^p(X)$ for $\Re(s)>d_k$ and $T^{[l]}(s)$ are linear operators $\H^*(M)\to\H^*(B,\H^l(F))$, which are meromorphic in $s$. In the context of mathematical scattering theory, the $T^{[l]}(s)$ are referred to as scattering operators.

In view of the Hodge-theorem the spectral value $s=2d_k$ is of particular interest, because if $E(s,\phi)$ is defined there,  it is harmonic. Thus the next task is to identify the poles of $s\mapsto E(s,\phi)$.

In the theory of automorphic forms, information about the poles of generalized eigenforms can be read off from product formulas known as the Maa{\ss}--Selberg relations. We will derive a similar formula for the inner product of GEs which are perpendicular to fiber-harmonic forms outside of a compact set (\textbf{Proposition \ref{maass-selb1}}).
This provides us with detailed information about the location and order of poles of $s\mapsto E(s,\phi)$. In particular, the order of a pole in $s=2d_k$ coincides with the maximal order of a pole of the scattering operator $T(s)$ and is at most one (section \ref{epole}).

Moreover the Maa{\ss}--Selberg relations allow us to derive important properties of the residues of the scattering operator and GEs: Let ${\tC}^{[k]}\defgleich \res_{s=2d_k}T^{[l]}(s) $ and $\tE(\phi)\defgleich\res_{s=2d_k}E(s,\phi)$.
Then $\H^{*,k}(M)\defgleich\H^*(B,\H^k(F))$ splits into the orthogonal direct sum
\[
 \H^{*,k}(M)=\ker {\tC}^{[k]}\oplus \bild {\tC}^{[k]},
\]
and the residue $\tE(\phi)$ is in $L^2\Omega^p(X)$.

For the middle fiber degree we have even more information about $\H^{*,f/2}(M)$. 
Again from the Maa{\ss}--Selberg relations and the functional equations for $E$ we can conclude that
$T^{[f/2]}(s)$ is regular at $s=2d_k=0$ and that
$T^{[f/2]}(0)$ is  selfadjoint and an involution on $\H^{*,f/2}(M)$. Thus 
$\H^{*,f/2}(M)=\H_+\oplus \H_-$, where $\H_\pm$ are the $\pm 1$-eigenspaces of $T^{[f/2]}(0)$.

Now we have all the information required for the classification of harmonic representatives of $H^p(X)$. We use
the following idea that goes back to G.~Harder.

Let $\phi\in\H^r(B,\H^k(F))$.
If the generalized eigenform $E(\cdot,\phi)$ does not have a pole at $s=2d_k$, then it is a solution of $\Delta E=0$ which is not square integrable.
Then it remains to determine under which conditions $E(2d_k,\phi)$ is closed.
For that, we employ the functional equations from
\textbf{Corollary \ref{funceq}} that are derived from the asymptotic expansion \eqref{introasym}.

If there is a pole at $2d_k$, we instead consider the residue, $\tE(\phi)$. This is a closed $L^2-$harmonic form due to \textbf{Corollary \ref{zerlres}}, thus a representative in both $H^p(X)$ and $H_{(2)}^p(X)$.

In this way to every $\phi\in \H^p(M)$ we may associate a so-called singular value, that is, a closed harmonic form $\Xi(\phi)\in \Omega^p(X)$. We interpret $\Xi$ as a linear map $H^p(M)\to H^p(X)$.
This is the map that is needed for Theorem \ref{thm1} and  the final step is to prove the  isomorphism $\Xi(\bild r)\simeq H_{\rm{inf}}$.

The image of the restriction map $r$ is of independent interest, in fact, the proof of Theorem \ref{thm1} relies heavily on its explicit description given in 
\begin{introenv}{Theorem}{thm2}
Let
\[
\mathcal{A}^p=\bigoplus_{k=0}^f \mathcal{A}^{(p-k,k)}\qquad\text{with}\qquad \mathcal{A}^{(p-k,k)}\defgleich
\begin{cases}
\bild {\tC}^{[k]}, & k<f/2\\
\H_+^p,&k=f/2\\
\astm\ker {\tC}^{[f-k]}, & f/2<k\\
0,& \text{otherwise}
\end{cases}
\]
Then under the Hodge-isomorphism
\[
\bild (r: H^p(X)\to H^p(M))\simeq \mathcal{A}^p.
\]
\end{introenv}

The detailed knowledge we have obtained about $H_{\mathrm{inf}}^p(X)$ also allows us to compute 
the signature of $X$ directly in \textbf{Proposition \ref{gleichesign}}, by showing that there are involutions $\tau$ on $L^2-$harmonic forms, which commute with the construction of $\tE$:
\[
 \tau_X \tE(\omega) = \tE(\tau_Z\omega).
\]
This recovers the identity 
$
L^2-\sign(X)= \sign(X_0,\partial X_0)
$
proven in \cite{dai}.

\subsection*{Acknowledgements}
This paper includes part of my Ph.D.~thesis.  I am indebted to my advisor W.~M\"uller for his mathematical guidance and deep insight. Also I am grateful to Eugenie Hunsicker, Alexander Strohmaier and Gregor Weingart for their interest and helpful discussions. 
Finally I would like to thank MSRI for hospitality during the program ``Analysis on Singular Spaces''.

\section{Manifolds with fibered cusp metrics}

\subsection{Manifolds with fibered cusp metrics}
\label{mfgspitz}

Let $(M,g^M)$ and $(B,g^B)$\index{$M$}\index{$B$} be closed, connected orientable Riemannian manifolds and $\pi: M\to B$  a  Riemannian submersion.
The fibers $F_b\defgleich\pi^{-1}(\{b\})$\index{$F$} are closed, and we assume that they are connected.
Since  $(M,g^M)$ is complete, $\pi: M\to B$ is a fiber bundle and all fibers are diffeomorphic, the diffeomorphism being given by horizontal lift of curves in $B$. In particular $f\defgleich\dim F$ is constant.

Let $TF$\index{$TF$} be the vertical tangent bundle, i.e.  $(TF)_y=T_y F_{\pi(y)}$. Let $g^F$ denote the family of Riemannian metrics on $TF$ induced by $g^M$.
Let $T^H M$\index{$T^H M$} be a horizontal distribution for $\pi:M\to B$, i.e.
$
 TM=T^H M\oplus TF.
$
By definition of a Riemannian submersion
$
T^H M \simeq \pi^* T B
$
is an isometry.

We define a family of metrics on $M$ by
\[
g_u^M =\pi^* g^B + e^{-2u} g^{F}, \qquad u\in\Reell^+,
\]
and equip $Z\defgleich\Reell^+\times M$ with the metric $g^Z\defgleich du^2+g_u^M$.\index{$Z$}

A manifold
$X$\index{$X$} is called  \emph{manifold with fibered cusp metric}, if $X$ is isometric to $Z$ outside of a compact set  $X_0$\index{$X_0$}. $X$ is a complete Riemannian manifold.

\subsubsection{Examples}
\begin{enumerate}[a)]
 \item 
 If the base $B$ is a point, the metric on $Z$ is
\[
du^2+e^{-2u} g,
\]
so that $Z$ is a cusp with base $M$ and $X$ is a manifold with cusp end.
Similarily, if $\pi=\text{id}$, then $X$ is a manifold with cylindrical end as considered in \cite{APS}.

\item  In \cite{mu-cusprk1}, \cite{mu-rank1}  W.~M\"uller considers locally symmetric spaces of $\Rational-$rank $1$. 
Here we consider only the case of a single cusp: Let $X=X_0\cup_M Z$, where $X_0$ is a compact manifold with boundary $M$, and $Z=\Reell^+\times M$ isometric
to a cusp of a locally symmetric space with  $\Rational-$rank $1$. 
Then there is a fiber bundle $M\to B$ where the fiber $F=\Gamma\cap N\backslash N$ is a compact nilmanifold, and $B$ is a locally symmetric space. The metric on $\Reell^+\times M$ then locally takes the form
\begin{equation}
 g^Z=du^2+\pi^*g^B+e^{-2 a u} g_1(b)+e^{-4 a u} g_2(b),\label{cprk1}
\end{equation}
where $a>0$ and $g^B$ is the metric on $B$. $g_1(b),  g_2(b)$ have support along the fibers $F_b$ over $b\in B$.
The volume form on $Z$ is given by
\[
 \vol_Z= e^{-q u} du\, \vol_B \vol_{F_b},
\]
for some $q>0$. If as the symmetric space we choose the $n$-dimensional hyperbolic space
$
SO_e(n,1)/SO(n)
$, then the metric \eqref{cprk1} takes the form
\[
g^Z=du^2+\pi^*g^B+e^{-2 a u} g_1(b),
\]
see Proposition 2.9 in \cite{weber} and e.g. \cite{carped}. 
A similar situation is considered in the work \cite{harder2}, there $X=\Gamma\backslash(\mathds{H}\times \ldots \times \mathds{H}\times Y\times\ldots \times Y)$ 
with the upper complex half plane $\mathds{H}$ and the 3-dimensional hyperbolic space $Y,$ and a torsion free conguence subgroup $\Gamma\subset SL(2,\mathfrak{O})$.
\end{enumerate}

Now we want to recall some facts about connections in the Riemannian fiber bundle $M\to B$, following
 \cite{bismut-lott}, pg. 323-329 and \cite{bismut-cheeger}, pg. 53-55.

As a bundle of $\Ganz-$graded algebras over $M$ we have the isomorphism
\begin{equation} \label{bm3.3}
 \Lambda(T^*M)\simeq \pi^*(\Lambda(T^*B))\otimes \Lambda(T^* F).
\end{equation}

Let $W$\index{$W$} be the smooth infinite dimensional $\Ganz-$graded vector bundle over $B$, whose fiber $W_b$ over  $b\in B$ is
$C^\infty(F_b, \Lambda(T^*F)|_{F_b})$.
This means
\[
 C^\infty(B,W^{(k)})\simeq C^\infty(M,\Lambda^k(T^*F))
\]
and we have an isomorphism of $\Ganz-$graded vector spaces
\[
 \Omega^\bullet(M)\simeq \Omega^\bullet(B,W)\defgleich C^\infty(B, \Lambda^\bullet(T^*B)\otimes W).
\]

%
Let $d^M$ be the outer derivative on $\Omega(M)$. Then as usual $(d^M)^2=0$. 
We consider the decomposition of $d^M$ in horizontal and vertical components.
Let
\[
 \Omega^{a,b}(M)\defgleich \Omega^a(B,W^{(b)})\;,\qquad \Omega^{*,k}\defgleich\bigoplus_{j=0}^{n-k} \Omega^{j,k},\quad n=\dim M
\]
and consider the decomposition $d^M=\sum_{i+j=1}d^{i,j}$
with
$
 d^{i,j}:\Omega^{a,b}\to\Omega^{a+i,b+j}.
$
The fiber-degree operator $\kappa$ is the linear operator \index{$\kappa$}$\kappa:\Omega^*(M)\to \Omega^*(M)$, which is defined by 
$
\kappa \phi= k \phi
$
for $\phi\in \Omega^{*,k}$.

In \cite[Proposition 3.4]{bismut-lott} and \cite[Proposition 10.1]{BGV} it is proven that
\begin{equation} \label{dmzerl}
d^M=d^{0,1}+d^{1,0}+d^{2,-1}
\end{equation}
where 

\begin{itemize}
\item $d^{0,1}=d^F\in C^\infty(B, \text{Hom}(W^\bullet,W^{\bullet+1}))$ is the differential along the fibers.
\index{$d^F$}

\item $d^{1,0}$ is given as follows. Let $X$ be a smooth vector field on $B$ with horizontal lift $X^H\in C^\infty(M, T^H M)$.
For $s\in C^\infty(B,W)$ and a vector field $X$ on $B$, we define a covariant derivative on $W$ that preserves the $\Ganz-$grading by
\[
 \nabla^W_X s= \text{Lie}_{X^H}s.
\]
$\nabla^W$  is extended in a unique way to a covariant derivative 
\[
   d^{1,0}:\Omega^\bullet(B,W)\to \Omega^{\bullet+1}(B,W),
\]
so that Leibniz' rule
\begin{equation}
 d^{1,0}(\alpha\wedge\theta)=d^B\alpha\wedge \theta +(-1)^r \alpha\wedge d^{1,0}\theta,\qquad \alpha\in \Omega^r(B),\theta\in \Omega^\bullet(B,W)\label{leibniz}
\end{equation}
holds.

\item Finally
\[
d^{2,-1}\in \Omega^2(B,\text{Hom}(W^\bullet,W^{\bullet-1}))
\]
is given by inner multiplication with curvature of the fibers, i.e. for a pair $(X,Y)$ of vector fields on $B$  
\begin{equation}
 d^{2,-1}(X,Y)(\omega)=i_{\Vert[Y^H,X^H]} \omega\label{defd21}
\end{equation}
where $\Vert$\index{$\Vert$} is the projection from $TM$ onto $TF$.
\end{itemize}

\subsection{Fibre-harmonic forms}\label{kfash}

Let $\ast$ be the  fiberwise Hodge-star-operator with respect to $g^F$. This is an operator on 
\[
C^\infty(M, \Lambda(T^*F))\simeq C^\infty(B,W).
\]
In this way $W$ obtains a Hermitean metric $h^W$, so that for $s, s'\in C^\infty(B,W)$ and $b\in B$
\begin{equation*}
 (\sce{s}{s'}_{h^W})(b) = \int_{F_b} s(b)\wedge \ast s'(b) = \int_{F_b} \sce{s(b)}{s'(b)}_{F_b} \vol_{F_b}.
\end{equation*}
Here $\sce{\cdot}{\cdot}_{F_b}$ is the scalar product on $\Omega^*(F_b)$ induced by $g^F$.
Using the metric $g^B$ on $B$, the hermitean fiber product $h^W$ extends to $\Omega^*(B,W)$.

\begin{defn}
Let $\delta^{i,j}$ be the fiberwise formal adjoint of $d^{i,j}$, and $\delta^F\defgleich\delta^{0,1}$.
\end{defn}\index{$\delta^F$}\index{$\delta^{i,j}$}


Let $H^*(F)=\bigoplus_{i=0}^f H^i(F)$ be the $\Ganz-$graded vector bundle over $B$, whose fiber over $b\in B$
is the cohomology $H^*(F_b)$ of the complex $(W_b, d^{F_b})$.

Let $V=\delta^F-d^F \in C^\infty(B, \text{End}(W))$. The Hodge-theorem gives an isomorphism
$
\H(F_b)\defgleich \ker(V_b)\cong H(F_b)
$
and this induces an isomorphism
\begin{equation} \label{h234}
\H^*(F)\defgleich \ker(V)\cong H^*(F)
\end{equation}
of $\Ganz-$graded finite dimensional vector bundles. 
\begin{defn}
An element in $\omega\in \Omega^*(B,\H^*(F))$ is called \emph{fiberwise harmonic}, or simply \emph{fiber-harmonic}.
\end{defn}
Let $\Delta^F\defgleich d^F\delta^F+\delta^F d^F$ be the Hodge-Laplace-operator along the fibers. Since the fibers of $M\to B$ are closed submanifolds, $\omega\in \Omega^*(B,W)$ is fiber-harmonic if and only if  $\Delta^F\omega=0$.

Since $\ker(V)$ is a subbundle of $W$,  it inherits a hermitean metric from $h^W$ by projection, and
with respect to $h^W$ there is a direct orthogonal decomposition
\begin{equation}
 \Omega^*(B,W)=\Omega^*(B,\H^*(F))\oplus \Omega^*(B,\H^*(F)^\perp)\label{faserorth}.
\end{equation}
Let $\Pi_0$\index{$\Pi_0$} resp. $\Pi_\perp$\index{$\Pi_\perp$} denote the projections onto $\Omega^*(B,\H^*(F))$ resp. $\Omega^*(B,\H^*(F)^\perp)$.
Then $\Pi_0 d^{1,0}$ is a connection on $\Omega(B,\H(F))$, which can be understood as a connection on
 $\Omega(B,H(F))$ by \eqref{h234}.

The following statement follows directly from Propositions 2.5 and 2.6 in \cite{bismut-lott}.
\begin{satz}\label{flachzus}
The connection $\nabla^{H(F)}=\Pi_0 d^{1,0}$ on $\Omega(B, H(F))$ is flat, i.e.  $(\nabla^{H(F)})^2=0$.
\end{satz}

Generally the decomposition \eqref{faserorth} is not invariant under $d^{1,0}$, i.e. $\Pi_\perp d^{1,0} \Pi_0 \neq 0$. We will give a sufficient criterion for invariance in Proposition \ref{satzproj} below.
\begin{defn}
Let\index{$H$} $\Hor:C^\infty(M, TM) \to C^\infty(M, T^HM)$ be the projection onto the horizontal distribution and $\{U_j\}_{j=1}^r\subset C^\infty(M, TF)$ be a vertical orthonormal frame. With the Levi-Civita connection  $\nabla^M$  on $C^\infty(M, TM)$ we define the \emph{mean curvature $H$ of the fibers} of $\pi:M\to B$ to be the horizontal vector field\index{$H$}
\[
H\defgleich\sum_{j=1}^{\dim F}\Hor\nabla^M_{U_j}{U_j}.
\] 
\end{defn}

Recall that a vector field on $M$ is called \emph{projectable}, if there is a vector field $\bH$ on $B$ such that $d\pi H=\bH$. A projectable horizontal vector field is called \emph{basic}.
\begin{satz}\label{satzproj}
If the mean curvature $H$ of the fibers is projectable, then  the decomposition \eqref{faserorth} is invariant under $d^{1,0}$ and $\delta^{1,0}$.
\end{satz}

\begin{proof}
By definition of $d^{1,0}$ it is sufficient to prove that $\nabla^W$ leaves
$
C^\infty(B,\H^*(F))\oplus C^\infty(B,\H^*(F)^\perp),
$ invariant,
then claim for $\Omega^*(B,W)$  follows from \eqref{leibniz}.

From $d^2=0$ we already know that $(d^2)^{1,1}=d^{1,0}d^F+d^F d^{1,0}=0$. To prove $\nabla^W C^\infty(B,\H^*(F))\subset C^\infty(B,\H^*(F))$ it remains to show
\[
 s\in C^\infty(B,\H^*(F))  \Longrightarrow \delta^F \nabla^W_X s=0
\]
for a basic vector field $X=\pi^* \bX$.

Let $\Phi_t^{\bX}$ be the flow of ${\bX}$.  The flow $\Phi_t^{X}$ of $X$ projects onto  $\Phi_t^{\bX}$ and is a diffeomorphism of the fibers:
\begin{equation}
\Phi_t^{X}: F_b\to F_{\Phi_t^{\bX}(b)}.\label{fluss}
\end{equation}
Let $\omega$ be the volume form of the fibers, i.e. $\omega$ is the volume form on $M$, such that $\omega|_{F_b}=\vol_{F_b}$ is the volume form on $F_b$. 
Then because of \eqref{fluss}
\begin{align*}
\bX(h^W(s_1, s_2))&=\text{Lie}_{\bX} \Big( \int_{F_b} \sce{s_1}{s_2}_{F_b} \vol_{F_b} \Big) \\
&= \int_{F_b} \text{Lie}_X (\sce{s_1}{s_2}_{F_b}) \vol_{F_b} +  \int_{F_b} \sce{s_1}{s_2}_{F_b} (\text{Lie}_X \omega)|_{F_b}
\end{align*}
Let $f_X= -g^M(H,X)\in C^\infty(M)$ with the mean curvature $H$ of the fibers. It is well known\footnote{Lemma 10.4 in \cite{BGV} or Theorem 6.6 in \cite{lang}} that
\[
 (\text{Lie}_X \omega)|_{F_b} = f_X\vol_{F_b}.
\]
Because  $\nabla^W$ is compatible with $\sce{}{}_{F_b}$,
\begin{equation}
 \bX(h^W(s_1,s_2)) 
= h^W(\nabla^W_X s_1,s_2)+h^W(s_1,\nabla^W_X  s_2) +h^W(f_X s_1, s_2)\label{bed6}
\end{equation}

Now let $s_2$ be fiber-harmonic, i.e. $\delta^F s_2=0,$ $d^F s_2=0$. Using $\nabla^W_X d^F= - d^F\nabla^W_X$, we obtain
\begin{align*}
 0&= \bX(h^W(s_1,\delta^F s_2)) 
= \bX(h^W(d^F s_1,s_2))\\
&= -h^W( \nabla^W_X s_1, \delta^F s_2)+h^W( s_1, \delta^F \nabla^W_X s_2)+ h^W(f_X  d^F s_1, s_2)\\
&= h^W( s_1, \delta^F \nabla^W_X s_2)+ h^W(d^F (f_X s_1), s_2)-h^W(d^F(f_X) s_1, s_2),
\end{align*}
so that
\[
 h^W( s_1, \delta^F \nabla^W_X s_2) = h^W(s_1,  d^F(f_X) s_2).
\]
In particular for $s_1= \delta^F \nabla^W_X s_2$,
\[
 h^W(s_1, s_1)
=  h^W(\nabla^W_X s_2 ,d^F(d^F(f_X) s_2))=0
\]
which implies $s_1=0$. This shows $\nabla^W C^\infty(B,\H^*(F))\subset C^\infty(B,\H^*(F))$.

Now let $s_1\in C^\infty(B,\H^*(F)^\perp),$ $s_2 \in C^\infty(B,\H^*(F))$. 
If $H$ is projectable with $\pi_* H={\bH}$, then $f_X=-\pi^*g^B(\bX,\bH)$ and
from \eqref{bed6}
\begin{align*}
 h^W(\nabla^W_X  s_1, s_2)&= h^W(\nabla^W_X  s_1, s_2)+h^W( s_1, \nabla^W_X  s_2)-g^B(\bX,\bH)h^W(s_1,s_2)\\
&=\bX(h^W(s_1, s_2)) = 0
\end{align*}
so that $\nabla^W C^\infty(B,\H^*(F)^\perp)\subset C^\infty(B,\H^*(F)^\perp)$.

The statement about $\delta^{1,0}$ finally follows from $\Pi_0 \delta^{1,0}\Pi_\perp=(\Pi_\perp d^{1,0}\Pi_0)^*$. 
\end{proof}


\subsection[Laplace-operator on $\Omega^*(Z)$]{Laplace-operator on $\mathbf{\Omega^*(Z)}$}\label{klap}

Let  $\Delta_Z$ be the Hodge-Laplace-operator on $Z$, where we consider   $\Delta_Z$ as linear operator on $L^2\Omega^p(Z)$ with domain 
$\dom \Delta_Z=\Omega_0^p(Z)$. Let $\bar\Delta_Z$  be the Friedrichs extension of $\Delta_Z$. 
More precisely, we define
$q(\phi,\psi)\defgleich(\phi, \Delta_Z\psi)_{L^2\Omega^p(Z)}$ with $\phi, \psi\in \dom \Delta_Z$. 
This is quadratic form with domain $Q(q)=\Omega_0^p(Z)$. Let $\tilde q$ the closure of this form. The Friedrichs extension $\bar \Delta_Z$ of $\Delta_Z$ then
is the selfadjoint operator on $L^2\Omega^p(Z)$ that is defined by $(\phi, \bar\Delta_Z\psi)=\tilde q(\phi,\psi)$. 
By definition the domain of the quadratic form of $\bar\Delta_Z$ is the closure of $\dom \Delta_Z=\Omega_0^p(Z)$ in the norm
\[
 \|\phi\|_1^2=\|\phi\|^2+q(\phi,\phi),
\]
i.e. the Sobolev-space $H_0^1\Omega^p(Z)$. The domain of $\bar\Delta_Z$ is $H^2\Omega^p(X)\cap H_0^1\Omega^p(Z)$. For $\varphi\in \dom \bar\Delta_Z\cap \Omega^p(\bar Z)$ we have
\begin{equation}\label{relrbd}
 i^*\varphi=0,\qquad i^*(\ast\varphi)=0
\end{equation}
with the inclusion $i:M\hookrightarrow Z$. From now on we write \index{$\Delta_Z$}$\Delta_Z$ for $\bar \Delta_Z$.

By\index{$\rho_u$}
$
 \rho_u(\omega)=e^{\kappa u} \omega
$
an isometry
\[
 \rho_u:(\Omega^*(M),g_u^M)\to (\Omega^*(M),g^M)
\]
is defined.

Let $\alpha, \beta$ be elements in  $\Omega^{*,k}(M)$, which depend on the real parameter $u$, i.e.
$\alpha, \beta \in \pi_2^*\Omega^{*,k}(M)$ for the projection $\pi_2: Z=\Reell^+\times M \to M$.

We extend  $\rho_u$ to an operator on $\Omega^\ast(Z)$ by setting $\rho_u^{-1}(du\wedge\beta)\defgleich du\wedge \rho_u^{-1}\beta$.
Also it will be  useful to consider the isometry\index{$\tilde\rho_u$}
\[
 \tilde\rho_u: L^2\Omega^*(Z,g_Z)\to L^2\Omega^*(Z,du^2+g),
\]
which is given by
\[
 \tilde\rho_u(\alpha+du\wedge\beta)\defgleich e^{(k-f/2)u}(\alpha+du\wedge\beta)
\]
Finally let\index{$a_k$}\index{$\ua$}\index{$d_k$}\index{$\ud$}
\begin{align*}
 a_k&=f/2-k, &d_k&=|a_k|\\
\ua&=f/2\cdot\text{id}-\kappa,& \ud&=|\ua|.
\end{align*}


The local form of  $\Delta_Z$ is described in 
\begin{satz}\label{s3.9}
For $\omega\in \Omega_0^*(Z)$ we write $\omega=(\alpha,\beta)\defgleich\alpha+du\wedge\beta$, where $\alpha, \beta \in \pi_2^*\Omega^{*}(M)$.
Moreover let
\begin{eqnarray*}
d_u^M&=&e^{-u}d^{2,-1}+d^{1,0}+e^{u}d^{F},\quad D_u^M=d_u^M+ \delta_u^M\\\
\tilde{T}_u&=& -{\textstyle\frac{\partial^2}{\partial u^2}}+\big({\textstyle\frac{f}{2}}-k\big)^2+(D_u^M)^2\\
Q_u&=&e^{-u}(d^{2,-1}-\delta^{2,-1})-e^{u}(d^{F}-\delta^{F}).
\end{eqnarray*}

With respect to the above decomposition
\begin{equation}\label{deltazf}
\tilde\rho_u(d+\delta)^2\tilde\rho_u^{-1} \omega =
\begin{pmatrix}
 \tilde{T}_u& -Q_u\\
Q_u &  \tilde{T}_u
\end{pmatrix}
\begin{pmatrix}
\alpha\\
\beta
\end{pmatrix}.
\end{equation}

\end{satz}
\begin{proof} 
Let $\astm$ be the Hodge-star-operator on $\Omega^*(M)$ with respect to $g^M$. 
Then \eqref{deltazf} follows with a straightforward calculation from $\Delta_Z=d^Z\delta^Z+\delta^Z d^Z$ and
\begin{align*}
d^Z(\rho_u^{-1} \alpha+du\wedge \beta) &=  \rho_u^{-1}(d_u^M\alpha+du\wedge(\partial_u\alpha-\kappa \alpha))-\rho_u^{-1} (du\wedge d_u^M\beta)\\
\delta^Z(\rho_u^{-1}\alpha+du\wedge \beta)&=\rho_u^{-1}\delta_u^M\alpha -\rho_u^{-1} (du\wedge \delta_u^M\beta+\partial_u\beta-(f-\kappa)\beta).\qedhere
\end{align*}

\end{proof}

\subsection{Cohomology and Harmonic Forms}\label{kohom}


In this chapter we introduce some additional notation and recall known results which will be referred to later.
Let $\Omega_0(X)$ denote the space of differential forms on $X$ with compact support.
As usual let $L^2\Omega^p(X)$ be the closure of $\Omega_0^p(X)$ in the norm induced from the scalar product
\[
 \Scr{\phi}{\psi}_{L^2\Omega^p(X)}\defgleich\int_X \phi\wedge \ast\psi.
\]

Most statements about cohomology in this article refer to de~Rham-cohomology\index{$H^p$}
\[
 H^p(X)=\frac{\{\omega\in \Omega^p(X)\mid d\omega=0 \}}{d \Omega^{p-1}(X)}
\]
and to  de~Rham-cohomology with compact support\index{$H_c^p$}
\[
 H_c^p(X)=\frac{\{\omega\in \Omega_0^p(X)\mid d\omega=0 \}}{d \Omega_0^{p-1}(X)}.
\]

Because $\Omega_0^p(X)$ is dense in $L^2\Omega^p(X)$, $d$ has a well-defined strong closure (again denoted by $d$), which is usually called the maximal closed extension of $d$.  Note that, since $X$ is complete in our case, all closed extensions of $d$ have the same domain due to a classical result of Gaffney.
The domain of the differential $d_p: L^2\Omega^p(X)\to  L^2\Omega^{p+1}(X)$ is 
\[
 \dom d_p=\big\{ \phi\in \Omega^p(X)\cap L^2\Omega^p(X) \mid d\phi\in L^2\Omega^{p+1}(X)  \big\},
\]
where $d\phi$ is understood in the distributional sense. 

We define the $p-$th $L^2-$cohomology group\index{$H^p_{(2)}$}
\[
 H^p_{(2)}=\frac{\ker d_p}{\bild d_{p-1}}
\]
and the $L^2-$harmonic $p-$forms\index{$\H_{(2)}^p$}
\[
\H_{(2)}^p(X)=\big\{ \omega\in L^2\Omega^p(X)\mid \Delta_X \omega\defgleich(d\delta +\delta d)\omega=0 \big\}.
\]
Here $\Delta_X$ is the selfadjoint extension of Laplace-operator on compactly supported forms $\Delta_X:\Omega_0^p(X)\to \Omega_0^p(X)$ to an operator on $L^2\Omega^p(X)$.
The regularity theorem for elliptic operators states that forms in $\H_{(2)}^p(X)$ are smooth.

For $\omega\in \H_{(2)}^p(X)$ we have $\scr{\Delta \omega}{\omega}=\|d\omega\|^2+\|\delta\omega\|^2=0$, so that
$\H_{(2)}^p\subset \ker d_p$, which induces a map
\[
\H_{(2)}^p\to H^p_{(2)},\qquad\omega\mapsto [\omega].
\]
In general this map is neither injective nor surjective. However, the space of $L^2-$harmonic forms is  isomorphic to the \emph{reduced} $L^2-$cohomology:
\[
  \H_{(2)}^p\cong H^p_{(2),\text{red}}=\frac{\ker d_p}{\;\overline{\bild d_{p-1}}\;}.
\]

Let $H^p_!(X)\defgleich \bild\big(H_c^p(X)\to H^p(X)\big)$. 
It is a well known result (e.g. \cite{anders}) that for a \emph{complete} Riemannian manifold $X$ there is a natural injective map
\[
H^p_!(X)\to H^p_{(2),\text{red}}(X).
\]
In particular every class in $H^p_!(X)$ has a unique $L^2-$harmonic representative.

Finally we want to mention the following theorem of Kodaira:
\begin{equation*}
 L^2\Omega^p(X)=\H_{(2)}^p(X)\oplus \overline{\delta \Omega_0^{p+1}(X)} \oplus \overline{d \Omega_0^{p-1}(X)}
\end{equation*}

\section{Spectral Theory}

\subsection[Point spectrum of $\Delta_Z$]{Point spectrum of $\mathbf{\Delta_Z}$}

The decomposition \eqref{faserorth} admits an extension to 
\begin{equation}
   L^2\Omega^p(Z)=\Pi_\perp L^2\Omega^p(Z)\oplus \Pi_0 L^2\Omega^p(Z) \label{decoml2}
\end{equation}
where $\Pi_0 L^2\Omega^p(Z)$ is the closure of compactly supported forms, which are fiberharmonic in each cross-section $\{u\}\times M\subset Z$.

In analogy to the theory of classical automorphic forms, we define
\begin{defn}\label{spifo}
The \emph{cusp forms on $Z$} are elements in
\[
 L_{\text{cusp}}^2\Omega^p(Z)\defgleich\{\omega \in L^2\Omega^p(Z)\mid\exists \lambda\ge 0: \Delta_Z\omega=\lambda\omega, \quad\Pi_0(\omega)=0 \}
\]
Here $\Delta_Z\omega$ first has to be understood in the sense of distributions, but from elliptic regularity we get the smoothness of cusp forms.

The significance of cusp forms in the setting of manifolds with fibered cusps comes from
\begin{satz}\label{kptr}
The restriction of $\Delta_Z$ to the orthogonal complement of fiber-harmonic forms
\[
 \Pi_\perp \Delta_Z \Pi_\perp: {\Pi_\perp L^2\Omega^*(Z)}\cap\dom\Delta_Z \to {\Pi_\perp L^2\Omega^*(Z)}
\]
has pure point spectrum.
\end{satz}
\begin{proof}
The proof uses the Minimax-principle \cite[Theorem XIII.1]{reedsimon4} for $A\defgleich\Pi_\perp \Delta_Z \Pi_\perp$.

This states that the real numbers
\[
\lambda_n(A) = \sup_{v_1,\ldots,v_{n-1}\in Q(A)} \inf_{\substack{v_n\in  H_0^{1}\\ v_n\perp v_k \forall k<n}} 
\frac{\norm{\tilde\rho d^Z\tilde\rho^{-1}v_n}^2+\norm{\tilde\rho\delta^Z\tilde\rho^{-1}v_n}^2}{ \|v_n\|^2}
\]
are either eigenvalues of $A$ or accumulate at the beginning of the essential spectrum of $A$. Here $Q(A)\subset H_0^1\Omega^p(Z)$ is the form domain of $A$. 

Because of Theorem XIII.64 in  \cite{reedsimon4},
the resolvent of $A$ is compact, if
\begin{equation}
\lim_{n\to\infty}\lambda_n=\infty.\label{ziel1}
\end{equation}
Let $\omega=\alpha+du\wedge \beta, \alpha,\beta\in \Omega_0^{*,k}(Z)$ with $\supp \omega \subset [a,\infty)\times M$ and $\|\alpha\|_{H^1}=1, \|\beta\|_{H^1}=1$;
$\alpha, \beta$ depend on the parameter $u\in \Reell^+$.

Let 
\[
 V_u\defgleich
\begin{pmatrix}
(D_u^M)^2& -Q_u\\
Q_u &  (D_u^M)^2
\end{pmatrix}
\]
so that
\[
 \tilde\rho\Delta_Z\tilde\rho^{-1}=(-{\textstyle \frac{\partial^2}{\partial u^2}}+(f/2-k)^2)\begin{pmatrix}
 1& 0\\
0 & 1
\end{pmatrix}
+ V_u,\qquad 
\]
as sum of quadratic forms.
We will show
\begin{equation}\label{gwz}
|\Scr{V_u\omega}{\omega}| \to\infty\quad\text{for}\quad a\to\infty.
\end{equation}
Then since  $-{\textstyle \frac{\partial^2}{\partial u^2}}+(f/2-k)^2$ has pure absolutely continuous spectrum $[(f/2-k)^2,\infty)$, a standard argument using the Minimax-principle shows \eqref{ziel1}.

In the following, all norms and scalar products are meant to be those in  $L^2\Omega^p(Z,du^2+g_0)$.
\begin{eqnarray*}
\Scr{V_u\omega}{\omega} &= &\Scr{\pfrac{(D_u^M)^2\alpha-Q_u\beta}{Q_u\alpha+(D_u^M)^2\beta}}{\pfrac{\alpha}{\beta}} \\
&=&\Scr{(D_u^M)^2\alpha}{\alpha}+\Scr{(D_u^M)^2\beta}{\beta}-\Scr{Q_u\beta}{\alpha}+\Scr{Q_u\alpha}{\beta}\\
&=& \norm{D_u^M \alpha}^2 + \norm{D_u^M \beta}^2 + 2 \Scr{Q_u \alpha}{\beta}\\[2ex]
 |\Scr{V_u\omega}{\omega}| &\ge& \norm{D_u^M \alpha}^2+ \norm{D_u^M \beta}^2 -2\norm{Q_u\alpha}\norm{\beta}
\end{eqnarray*}

Because $\|\alpha\|_{H^1}=1$, all $\|d^{i,j}\alpha\|, \|\delta^{i,j}\alpha\|$ are bounded by a constant independent of $a$, and the same holds for $\beta$:
\begin{equation*}
 |\Scr{V_u\omega}{\omega}|\ge \norm{D_u^M \alpha}^2+ \norm{D_u^M \beta}^2 -C_0(\norm{e^u d^F\alpha}+\norm{e^u\delta^F\alpha})+C_1.\label{xyz888}
\end{equation*}
A simple calculation shows
\begin{equation*}
 \norm{D_u^M \alpha}^2 \ge \norm{e^u d^F\alpha}^2+\norm{e^u \delta^F\alpha}^2-C_2\norm{e^u d^F\alpha}
 -C_3 \norm{e^u \delta^F\alpha}+C_4, \label{xyz777}
\end{equation*}
and an analogous estimate holds  for $ \norm{D_u^M \beta}^2$. 

But now $\norm{e^u d^F\alpha}\ge e^{a}\norm{d^F\alpha}$. Thus for every real constant $c$
\[
 \norm{d^F\alpha}\neq 0\quad\Rightarrow\quad \lim_{a\to\infty} \big(\norm{e^u d^F\alpha}^2-c \norm{e^u d^F\alpha}\big) = +\infty,
\]
and so we conclude from \eqref{xyz888} and \eqref{xyz777} that \eqref{gwz} holds, i.e.
\[
 |\Scr{V_u\omega}{\omega}| \to\infty\quad\text{for}\quad a\to\infty,
\]
if $\norm{d^F\omega}\neq 0$ or $\norm{\delta^F\omega}\neq 0$. 
This proves the claim.
\end{proof}

\end{defn}

\subsection[Essential spectrum of $\Delta_Z$]{Essential spectrum of $\mathbf{\Delta_Z}$}
With respect to the decomposition \eqref{decoml2}
we can write
\[
 \tilde\rho\Delta_Z\tilde\rho^{-1}= \mathfrak{L}+
\begin{pmatrix}
0 & \Pi_\perp \Delta_Z \Pi_0 \\
\Pi_0 \Delta_Z \Pi_\perp & 0\end{pmatrix}
\quad\text{with}\quad
 \mathfrak{L}\defgleich\begin{pmatrix}
\Pi_\perp\Delta_Z \Pi_\perp & 0 \\[1ex]
0 & \Pi_0\Delta_Z \Pi_0\\
\end{pmatrix}.
\]
Here the off-diagonal terms
\begin{equation*}
\Pi_\perp \Delta_Z \Pi_0 = 
\Pi_\perp \begin{pmatrix}
(D_u^M)^2 & -Q_u \\
Q_u & (D_u^M)^2
\end{pmatrix}
\Pi_0
\quad\text{and} \quad\Pi_\perp \Delta_Z \Pi_0 
\end{equation*}
are bounded operators in $L^2\Omega^p(Z)$.

Now  we want to examine the contribution of fiberharmonic forms to the spectrum of $\Delta_Z$.
It is easy to see
that $\Pi_0\tilde\rho\Delta_Z\tilde\rho^{-1}\Pi_0$
is a relatively compact perturbation of 
$
 -{\textstyle \frac{\partial^2}{\partial u^2}}+(f/2-\kappa)^2,
$
and as such has pure absolutely continuous spectrum. Together with Proposition \ref{kptr} we get

\begin{satz}\label{wspek}
$\Delta_Z$ and $\mathfrak{L}$ have the same essential spectrum.
\end{satz}
\begin{proof}
The proof is similar to Theorem 2 in \cite{lott}, see \cite{jmthesis}
\end{proof}

\subsection{Two conditions}\label{Bedingungen}
For the remaining chapters we will make the assumptions
\begin{enumerate}[(A)]
 \item The horizontal distribution is integrable, i.e. $d^{2,-1}=0$.
\item $\Pi_\perp \delta^{1,0} \Pi_0 =0$
\end{enumerate}
An immediate consequence of (A) is 
\[
(d^{1,0})^2=d^{2,-1}d^{0,1}+(d^{1,0})^2+d^{0,1}d^{2,-1}=({(d^M)}^2)^{2,0}=0
\]
so that
\[
\Delta_{1,0}\defgleich d^{1,0}\delta^{1,0}+\delta^{1,0}d^{1,0}=(d^{1,0}+\delta^{1,0})^2.
\]
From the Hodge-decomposition we have $[\Pi_0, d^F]=[\Pi_0 ,\delta^F]=0$.
Together with condition (B) this implies $[d^Z, \Pi_0]=[\delta^Z, \Pi_0]=0$,
thus $\Delta_Z$ leaves the splitting of $L^2\Omega^*(Z)$ into fiber-harmonic forms and their orthogonal complement invariant. Also
 $\Pi_0\Delta_Z\Pi_0$  takes the especially simple form
\[
\tilde\rho_u\Pi_0(d^Z+\delta^Z)^2\Pi_0\tilde\rho_u^{-1}=(-\partial_u^2+(f/2-\kappa)^2 +\Delta_{1,0})\Pi_0
\]
in this case.

\begin{lem}\label{delta10ell}
Under the given conditions (A) and (B),
\[
\Delta_{1,0}=d^{1,0}\delta^{1,0}+\delta^{1,0}d^{1,0}:\Omega^*(B,\H^k(F))\to \Omega^*(B,\H^k(F))
\]
is a non-negative symmetric elliptic operator.
\end{lem}
\begin{proof}
From Proposition \ref{flachzus} we get $\Delta_{1,0}=(d^{1,0}+\delta^{1,0})^2$, which is non-negative. 
From the local formulas for $\nabla^W$ and $(\nabla^W)^*$ in \cite{bismut-lott} (Proposition 3.5 and 3.7 there; also see \cite{gilkeybook}) 
one concludes that $\Delta_{1,0}$ is a generalized Laplace-operator; as such it is elliptic.
\end{proof}

A sufficient criterion for condition (B) is given by Proposition \ref{satzproj}. In particular (B) is fulfilled
when the fibers of $\pi: M\to B$ are minimal, or in the case of a warped product.
A necessary condition for (A) is given by
\begin{lem}[{\cite[Lemma 1.7.2]{gilkeybook}}]
If the horizonal distribution of a Riemannian submersion $\pi:M\to B$ is integrable, there are local coordinates $m=(y,b)$ in $M$, so that  $\pi(m)=b$. In these coordinates the metric on $M$ takes the form
\begin{equation}
 g^M= h_{ij}(b)db^i\otimes db^j+ f_{\alpha\beta}(y,b)dy^\alpha\otimes dy^\beta \label{prodmed1}
\end{equation}
If in addition $B$ is simply connected, then $\pi:M\to B$ is a global product with metric \eqref{prodmed1}. 
\end{lem}

%
%
%

\textbf{Example.\hspace{1em}}
Let $G/K$ be a symmetric space of non-compact type, where $K\subset G$ is a maximal compact subgroup of the non-compact semisimple Lie group $G$.
More precisely  $G$  is the group of real points of a semisimple algebraic group $\mathbf{G}\subset GL(n,\Complex)$, which is defined over $\Rational$.
Let $\Gamma\subset G$ be an arithmetic lattice\footnote{see e.g. \cite{weber}, Definition 2.1} of $\Rational-$rank $1$ so that $X=\Gamma\backslash G/K$
is a locally symmetric space of  $\Rational-$rank $1$.

This situation is also  considered in the articles \cite{harder} and \cite{harder2}.
Let $\mathbf{P}$ be a rational parabolic subgroup of $\mathbf{G}$ and $P=\mathbf{P}(\Reell)$. 
Then $P$ is a parabolic subgroup of $G$ and there is a  ``rational horocyclic decomposition'' (\cite[S.141]{weber})
\[
G/K\cong A_P\times N_P\times  X_P.
\]
Furthermore $\Gamma\cap P$ induces a
discrete group $\Gamma_P$, that operates on $X_P$. The ``cusp'' $Z$ corresponds to
\[
(\Gamma\cap P)\backslash (G/K)\cong A_P\times (\Gamma\cap P)\backslash 
(N_P\times X_P)
\]
The base of the cusp is $M=(\Gamma\cap P)\backslash(N_P\times X_P)$ and
the canonical projection $N_P\times X_P\to X_P$ induces a fibration
$
M\to B,
$
with $B=\Gamma_P\backslash X_P$.
Proposition 2.9 in \cite{weber} or Proposition 4.3 in \cite{borel} describe the local form on $M$, and it follows that the horizontal distribution of $\pi:M\to B$ 
is integrable. Furthermore Borel shows in the proof of \cite[Corollary 4.4]{borel} that horizontal parallel transport preserves the volume form of the fibers of $\pi$.
Under the assumption that the horizontal distribution is integrable, Lemma 10.4 in \cite{BGV} shows
\begin{equation}
 d^{1,0}\vol_{F_b}= -H^\vee \wedge \vol_{F_b}, \quad H^\vee \text{ dual to $H$ w.r.t. $g^M$},\label{hint1}
\end{equation}
so the fibers of $\pi$ are minimal and (B) is fulfilled.\qed

\subsubsection{Spectral sequence}\label{kspecseq}
At the end of this chapter we want to prove an important decomposition of the de~Rham-cohomology $H^p(M)$.

\begin{satz}\label{satzspec}
If the mean curvature $H$ of the fibers is projectable and  $\pi:M\to B$ is flat, then\index{$\H^a(B,\H^b(F))$}
\[
 H^r(B,\H^s(F))\cong \H^r(B,\H^s(F))\defgleich\{\omega\in \Omega^r(B,\H^s(F))\mid \Delta_{1,0}\omega=0\}
\]
and
\[
 H^p(M)\cong\bigoplus_{r+s=p} \H^r(B,\H^s(F)) 
\]
\end{satz}
\begin{proof}
The proof uses the Leray-Serre spectral sequence. 
For $0\le n\le\dim M$ let
\[
\FF_i \Omega^n(M)\defgleich \big\{\omega\in \Omega^n(M) \mid \omega(Y_1,\ldots, Y_n)=0,\text{ if }n-i+1\\
\text{ of the $Y_l$ are vertical}\big\}.
\]
and $\FF_{q}\defgleich \FF_0$ for $q<0$, $\FF_q=0$ for $q>n$.

A form $\omega\in \FF_i \Omega^n(M)$ can be expressed as sum of elements $\pi^*\eta^{(k)} \wedge \psi^{(n-k)}$ with
$\eta^{(k)}\in \Omega^k(B), \psi^{(n-k)}\in \Omega^V(M)$ for $k\ge i$. Here $ \psi\in\Omega^V(M)$ by definition means $X\lrcorner \omega=0$ for all horizontal vector fields $X$.

The $\FF_i \Omega^n(M)$ define a filtration 
\[
 \Omega^n=\FF_0\supset \FF_1 \supset \ldots \supset \FF_n \supset \FF_{n+1}=0,
\]
of $\Omega^n$ which is compatible with $d^M$, i.e. $d^M(\FF_i\Omega^n)\subset \FF_i\Omega^{n+1}$. This filtration gives rise to a spectral sequence $E^{p,q}$ as usual, see e.g. \cite{mccleary}, \cite{bott-tu}. In \cite{dai}  the first terms of this spectral sequence were calculated, with the result
\[
 E_0^{i,n-i}=\Omega^i(B,W^{(n-i)}),\qquad E_1^{i,n-i}=\Omega^i(B,H^{n-i}(F)).
\]
The Hodge theorem then gives the identification
\begin{equation*}
 E_1^{i,n-i}=\Omega^i(B,\H^{n-i}(F))
\end{equation*}
of $E_1$ with fiber-harmonic forms.
Finally for $p+q=n$
\begin{equation*}
E_2^{p,q}=H^p(B,\H^q(F)),
\end{equation*}
and under the further condition $d^{2,-1}=0$,  Dai shows that the spectral sequence degenerates at $E_2$, so that
\[
E_\infty^{p,q}=E_2^{p,q}=H^p(B, H^q(F)).
\]

Under the assumption of projectable mean curvature of the fibers, there is a Hodge decomposition of $\Omega^r(B,\H^s(F))$ with respect to  $\Delta_{1,0}$,
\[
 \Omega^{r,s}=\H^r(B,\H^s(F))\oplus\ker d^{1,0}\oplus \im \delta^{1,0}.
\]
This proves the first statement in Proposition \ref{satzspec}, and the second statement follows from $ H^{p}(M)=E_\infty^{p}=\bigoplus_{r+s=p} E_\infty^{r,s}$.
\end{proof}

\subsection[Spectrum of the Laplacian on $X$]{Spectrum of the Laplacian on $\mathbf{X}$}\label{kstreu}

As before let $\Delta_X:\Omega_0^p(X)\to \Omega_0^p(X)$ be the Laplace-operator on compactly supported forms. $\Delta_X$ admits an extension to a selfadjoint operator on $L^2\Omega^p(X)$, which we will again denote by $\Delta_X$.

Let $\Delta_{X;D}$ be the Friedrichs extension of $\Delta_{X}:\Omega_0^p(X\setminus (\{0\}\times M))\to \Omega_0^p(X\setminus (\{0\}\times M))$ to $L^2\Omega^p(X)$. This decomposes as $\Delta_{X;D}=\Delta_{X_0;D}\oplus \Delta_{Z}$, where $\Delta_Z$ was defined in section \ref{klap} and $\Delta_{X_0;D}$ is the Friedrichs extension of the Laplacian on $\Omega_0^p(X_0)$. $\Delta_{X_0;D}$ is a selfadjoint elliptic operator on a compact manifold with boundary, as such it has pure point spectrum.

We recall some results from mathematical scattering theory.
The \emph{wave operators} $W^\pm$ are defined by
\begin{equation}
 W^\pm(\Delta_{X},\Delta_{X;D})=s-\lim_{t\to\mp\infty} e^{i \Delta_X t} J e^{-i \Delta_{X;D} t} P_{ac}(\Delta_{X;D})\label{wof}
\end{equation}
with the inclusion $J:\dom \Delta_{X;D}\hookrightarrow \dom \Delta_{X}$ and the projection  $P_{ac}$ onto the absolutely continuous subspace of $\dom\Delta_{X;D}$ in $L^2\Omega^p(X)$. 
It is well known that if the wave operators exist, they are partial isometries $ W^\pm: P_{ac}(\Delta_{X;D})\to \bild W^\pm$, i.e. isometries on the complement of $\ker W^\pm$.
In this case the $W^\pm$ are called \emph{complete}, if $\bild W^\pm=\bild P_{ac}(\Delta_X)$. Then $W^\pm: \bild P_{ac}(\Delta_{X;D})\to \bild \Delta_{X}$ are unitary equivalences, in particular $\Delta_{X;D}$ and $ \Delta_{X}$ have the same absolutely continuous spectrum.

There are several methods to show existence and completeness of $W^\pm$. In the present case most information can be obtained from the \emph{Enss method}, see e.g. \cite{BaWo} and \cite{mu-cusprk1}.
If we assume conditions (A) and (B) from section  \ref{Bedingungen}, the reasoning in the case of the manifold $X$ with fibered cusp metric is analogous to that in chapter 6 of  \cite{mu-cusprk1}. Not only does the Enss method show existence and completeness of the wave operators, but also that $\Delta_X$ has empty singular continuous spectrum, and that the point spectrum of $\Delta_X$ has no points of accumulation outside of $\text{spec}(\Delta_{X;D})$ 

Altogether this shows 
\begin{satz}\label{acspec1}
The absolutely continuous part of $\dom\Delta_X$  is unitarily equivalent to the fiber-harmonic forms $\Pi_0L^2\Omega^p(Z)$.
\end{satz}

\subsection{Analytic continuation of the Resolvent}\label{parakapform}
\subsubsection{Parametrix}
Under the conditions from chapter \ref{Bedingungen} the construction of a parametrix for the resolvent of the Laplace-operator
$\Delta_X$ on $L^2\Omega^*(X)$ is entirely parallel to \cite{mu-cusprk1} or \cite{mu-rank1}. We recall only the essential steps.

Let $X_1=X_0\cup([0,1]\times M)$ and let $\hat X$ be a closed manifold into which $X_1$ is embedded isometrically. For the Hodge-Laplace-operator $\Delta_{\hat X}$ on  $\hat X$
the resolvent $(\Delta_{\hat X}-\lambda)^{-1}$ is an operator valued function which is  meromorphic in
 $\lambda\in \Complex$ with poles in the eigenvalues of $\Delta_{\hat X}$. Let $Q_1(x,x',\lambda)$
be the restriction of the resolvent kernel of $(\Delta_{\hat X}-\lambda)^{-1}$ on $X_1\times X_1$.
This is the ``inner'' part of the parametrix. The ``outer'' part is given by the resolvent of  $\Delta_Z$, i.e. $Q_2(\lambda)=(\Delta_Z-\lambda)^{-1}$. The latter will be examined further now.

Let
\[ 
\Delta_{\text{cusp}}=\Pi_\perp\Delta_Z\Pi_\perp=\Delta_Z\Pi_\perp,\qquad \Delta_{1;Z}=\Pi_0\Delta_Z\Pi_0=\Delta_Z\Pi_0
\]
be the restrictions of $\Delta_Z$ on $L_{\text{cusp}}^2\Omega^p(Z)$ resp. $\Pi_0 L^2\Omega^p(Z)$. 

Since the splitting
\[
 L^2\Omega^p(Z)=L_{\text{cusp}}^2\Omega^p(Z)\oplus \Pi_0 L^2\Omega^p(Z)
\]
into cusp forms (Definition \ref{spifo}) and their orthogonal complement is invariant under $\Delta_Z$,
we obtain for the resolvent
\[
Q_2(\lambda)=(\Delta_Z-\lambda)^{-1}=(\Delta_{\text{cusp}}-\lambda)^{-1}+(\Delta_{1;Z}-\lambda)^{-1}.
\]
According to Proposition \ref{kptr} the resolvent of $\Delta_{\text{cusp}}$ is compact. It has the kernel
\[
 K_{\text{cusp}}(\lambda,x_1,x_2)=\sum_i \frac{1}{\lambda_i-\lambda} \psi_i(x_1)\otimes\overline{\psi_i(x_2)}
\]
for eigenforms $\psi_i$ of $\Delta_{\text{cusp}}$ to the eigenvalue $\lambda_i$.

Let
\[
 \tau_0=\inf\big\{d_k^2+\mu\mid 0\le k\le f,
 \mu\in\text{spec} \{\Delta_{1,0}:\Omega^{p-k}(B,\H^k(F))\to\Omega^{p-k}(B,\H^k(F))\} \big\}.
\]
The spectrum of $\Delta_{1;Z}$ is $[\tau_0,\infty)$ with branch points at $(f/2-k)^2+\nu, 0\le k\le f $ for each eigenvalue 
 $\nu$ of $\Delta_{1,0}$.
Certainly $\Complex\smallsetminus \Reell^+$ is contained in the resolvent set of $\Delta_{1;Z}$.
We want to compute the integral kernel of $(\Delta_{1;Z}-\lambda)^{-1}$ for $\lambda\in \Complex\smallsetminus \Reell^+$
explicitly.

With respect to the decomposition $\Omega^p(Z)=\pi_2^*\Omega^p(M)\oplus \pi_2^*\Omega^{p-1}(M)$ we have (chapter \ref{Bedingungen})
\[
\tilde\rho_u\Delta_{1;Z}\tilde\rho_u^{-1}=\begin{pmatrix}
 \check{T}_u& 0\\
0&  \check{T}_u
\end{pmatrix} ,
\quad  \check{T}_u= -\partial_u^2+(f/2-\kappa)^2 +\Delta_{1,0}.
\]
A simple calculation shows that the integral kernel  $K_1^{(p)}$ of $(\tilde\rho_u\check{T}_u\tilde\rho_u^{-1}-\lambda)^{-1}$ for $\lambda\in\Complex\setminus\Reell$  is given by
\begin{multline} \label{reskerf}
K_1^{(p)}(\lambda,(u,y),(r,z)) \\ = \sum_k\sum_{\mu^{(k)}}^\infty \frac{i}{2} e^{a_k (u+r)}\frac{ e^{i |u-r| \sqrt{\lambda-d_k^2-\mu^{(k)}} }- e^{i (u+r) \sqrt{\lambda-d_k^2-\mu^{(k)}} }}{\sqrt{\lambda-d_k^2-\mu^{(k)}}}
\:(\phi_\mu^{(k)})(y) \otimes (\overline{\phi}_\mu^{(k)})(z),
\end{multline}
where $\phi_\mu^{(k)}\in \Omega^{p-k}(B, \H^k(F))$ form a local orthonormal basis for every $k$, with $\Delta_{1,0} \phi_\mu^{(k)}=\mu^{(k)} \phi_\mu^{(k)}$.

Finally
$
 K_{\text{cusp}}+K_1^{(p)}+du\wedge K_1^{(p-1)}
$
is the integral kernel of $(\Delta_Z-\lambda)^{-1}$.

Let $\xi_1,\xi_2,\chi_1,\chi_2$ be suitable cut-off functions with the properties as in e.g. \cite{APS}.
Then we define an operator $Q(\lambda):L^2\Omega^p(X)\to L^2\Omega^p(X)$ by its integral kernel
\[
Q(\lambda,x_1,x_2)= \chi_1(x_1) Q_1(\lambda,x_1,x_2) \xi_1(x_2)+\chi_2(x_1) Q_2(\lambda,x_1,x_2) \xi_2(x_2).
\]

From this formula we get as in  \cite{mu-rank1} that
\[
Q(\lambda)(\Delta_X-\lambda)=\text{Id}+\mathcal{K}(\lambda),\qquad\lambda\in \Complex \setminus\Reell^+
\]
where $\mathcal{K}(\lambda)$ are compact operators in $L^2\Omega^p(X)$, and that
\[
 (\Delta_X-\lambda)^{-1}-Q(\lambda)
\]
for $\lambda\in \Complex \setminus\Reell^+$ is a family of compact operators which is meromorphic in  $\lambda$. 
Altogether this shows that $Q(\lambda)$ is a parametrix for $(\Delta_X-\lambda)^{-1}$.

\subsubsection{Continuation to the Spectral Surface}\label{spektrfl}

The square roots $\lambda\mapsto \sqrt{\lambda-d_k^2-\mu_j}$ in \eqref{reskerf} are holomorphic in the complex plane $\Complex\setminus[\tau_0,\infty)$ and
they cannot be extended  holomorphically to the whole complex plane. 

Let $\mathfrak{I}=\{\tau_i\}_{i\in\Nat}$ a discrete set of real numbers with $-\infty<\tau_0<\tau_1<\ldots$.
We recall the construction of the \emph{spectral surface} $\Sigma_s$. This is a Riemannian surface on which all square roots $z\mapsto\sqrt{z-\tau_i}$ are holomorphic.
\index{$\Sigma_s$}As in \cite{guil} we define
\[
\Sigma_s=\left\{\Lambda=(\Lambda_\mu)\in \Complex^{\sharp \mathfrak{I}} \mid \forall \mu, \nu \in \mathfrak{I}: \Lambda_\mu^2+\mu=\Lambda_\nu^2+\nu \right\}
\]
and the projection $\pi_s:\Sigma_s\to \Complex$, \index{$\pi_s$}$\pi_s(\Lambda)=\Lambda_\mu^2+\mu$. Then $(\Sigma_s,\pi_s)$ is a covering of $\Complex$ with infinitely many leafs and branch points in $\mathfrak{I}$. For $\mu\in\mathfrak{I}$ the square roots are given by
\begin{equation}\label{wufu}
\sqrt{\Lambda-\mu} \defgleich\Lambda_\mu.
\end{equation}
We define the \emph{physical domain} \index{$\FP$}$\FP$, i.e. the leaf of ``positive'' square roots, as
\[
\FP\defgleich\left\{\Lambda\in \Sigma_s \mid \forall \mu \in \mathfrak{I}: \text{Im } \Lambda_\mu>0\right\}.
\]
This can be identified with  $\Complex\setminus [\tau_0,\infty)$ via $\pi_s$.
The boundary of $\FP$ consists of two rays $\partial_\pm \FP \simeq [\tau_0, \infty)$. Here $\Lambda\in \partial_\pm \FP$ means that
$\pi_s^{-1}(\pi_s(\Lambda)\pm i\eps)  \to \Lambda$ for $\eps\to 0$.

%
Let $\gamma_{\tau}:\Sigma_s\to\Sigma_s$ be a deck transform, which is given by a closed curve which encircles  
the single branch point $\tau$.
Then square roots \eqref{wufu} fulfill the following identities:
\begin{equation}\label{wurzelverh}
 \sqrt{\gamma_\mu\Lambda-\mu}=-\sqrt{\Lambda-\mu},\qquad  
\sqrt{\gamma_\mu\Lambda-\nu}=\sqrt{\Lambda-\nu},\quad \mu\neq\nu
\end{equation}
Finally let $\Sigma_s^\mu$ for $\mu\in \mathfrak{I}$ be the connected component of $\FP$ in $\pi_s^{-1}(\Complex\setminus[\mu,\infty))$.

Now in our case 
\[
\mathfrak{I}=\{\mu+d_k^2\mid \mu\in\text{spec}(\Delta_{1,0}:\Omega^*(B, \H^{k}(F))\to\Omega^*(B, \H^{k}(F)),\quad 0\le k\le f \}.
\]
On the corresponing spectral surface $\Sigma_s$ all square roots in \eqref{reskerf} are holomorphic functions.

Outside $\FP$ the integral kernel in \eqref{reskerf} does not define a continuous operator on $L^2\Omega^p(X)$. 
However we can introduce weighted $L^2-$spaces, on which the analytic continuation of the kernel will be a continuous operator.
For $\delta\in\Reell$ consider the weight operator $\omega_\delta$ on differential forms $\phi\in \Omega^p(Z)$ defined by
\[
\omega_\delta(\phi)(u,y)\defgleich e^{\delta u} \phi(u,y)
\]
The corresponding weighted $L^2-$space is
\[
L_\delta^2\Omega^p(Z)=\left\{\phi: Z\to\Reell\mid \phi\text{ measurable and }
\omega_\delta(\phi) \in L^2\Omega^p(Z)\right\}
\]
For $\delta>0$ we have
\[
L_\delta^2\Omega^p(Z) \subset L^2\Omega^p(Z) \subset L_{-\delta}^2\Omega^p(Z).
\]

For $\tau>0$ let $D_\tau(0)=\{z\in \Complex \mid |z|<\tau \}$ and $\tilde{\mu}(\tau)$
the smallest eigenvalue of $\Delta_{1,0}$ with $\tilde{\mu}(\tau)+d_k^2>\tau$. Let
\[
\Omega_\tau = ( \FP \cup \pi_s^{-1}(D_\tau(0)) ) \cap \Sigma_s^{\tilde{\mu}(\tau)},
\]
and choose $\delta>0$ with $\delta^2>\tau$. Then $|\Im \sqrt{\Lambda-d_k^2-\nu}|<\delta$ for $\Lambda\in \Omega_\tau$ and $\nu\le\tilde{\mu}(\tau)$. This implies
\begin{lem}
For all $\eps>0$ and $\delta>0$ with $\delta^2>\eps$ the parametrix $Q(\lambda)$ has a continuation to a meromorphic family of continuous operators $Q(\Lambda): L_\delta^2\Omega^p(X)\to L_{-\delta}^2\Omega^p(X),
\Lambda \in \Omega_\eps$.
\end{lem}

For our later applications it is sufficent to know the analytic continuation of the resolvent in a neighbourhood of $\lambda=0$, i.e. the continuation to $\Sigma_s^{\tau_1}$, the Riemann surface for $z\mapsto\sqrt{z}$.
Let
\[
\tau_1=\min (\mathfrak{I}\smallsetminus \{0\})
\]
and $0<\eps<\tau_1$.
If 
\[
 \Lambda\in \Omega_\eps \defgleich\big(\FP\cup \pi_s^{-1}(B_\eps(0))\big)\cap \Sigma_s^{\tau_1},
\]
then all coefficients in the resolvent kernel \eqref{reskerf} lie in $L_{-\alpha}^2(\Reell^+,du)=e^{\alpha u} L^2(\Reell^+,du)$ with $\alpha^2>\tau_1$.
Thus for $\Lambda\in\Omega_\eps$ the resolvent kernel defines a continuous operator $L_{\alpha}^2\to L_{-\alpha}^2$.


As in \cite{mu-cusprk1} 
 we obtain
\begin{thm}\label{resfortform}
For  $\delta^2>\eps>0,$ $\eps<\tau_1,$ the resolvent $(\Delta-\lambda)^{-1}$ can be continued analytically to a family of operators 
$R(\Lambda)\in \mathscr{B}(L_\delta^2\Omega^p(X), L_{-\delta}^2\Omega^p(X))$ which is meromorphic in $\Lambda\in \Omega_\eps$.
\end{thm}
\begin{bem}
A similar statement holds for any $\eps>0$, then $\delta$ must be chosen sufficiently large.
\end{bem}

Since $\Omega_0^p(X)\subset L_a^2\Omega^p(X)\subset L^2_{\text{loc}}\Omega^p(X)$ for all $a\in\Reell$ und $\Sigma_s=\bigcup_{\eps>0} \Omega_\eps$, from the above theorem follows
\begin{cor}
The resolvent $(\Delta-\lambda)^{-1}$ has an analytic continuation to a family of operators $R(\Lambda):\Omega_0^p(X)\to L^2_{\text{\upshape loc}}\Omega^p(X)$ which is meromorphic in $\Lambda\in \Sigma_s$.
\end{cor}

\subsection{Generalized Eigenforms}\label{keif}

The generalized eigenforms provide the spectral resolution of the absolutely continuous part of $\dom(\Delta_X)$. Once the analytic continuation of the resolvent of $\Delta_X$ is known, it is possible to construct generalized eigenforms of $\Delta_X$ explicitely as follows.

An element in $\Omega^{l,k}=\Omega^l(B,\H^k(F))$ is a section in the finite dimensional vector bundle
$W_0^{k+l}=\Lambda^l T^* B\tensor \H^k(F)$ over $B$. There is an orthonormal basis $\{\phi_\mu\}$
of $L^2(B, W_0^{k+l})$ consisting of eigenforms of $\Delta_{1,0}$.
Let $z\mapsto\sqrt{z}$ be the branch of the square root, for which $\Im \sqrt{z}>0$ if $z\in\Complex\setminus\Reell^+$.
Let $\psi\in \Omega^{*,k}$ fiber harmonic with degree $k$ in the fibers, so that 
$
\Delta_{1,0}\psi=\mu\psi.
$
Let
\[
 e_{\mu,k,\pm}(\lambda,\psi,(u,x))\defgleich\tilde\rho_u^{-1} e^{\pm i\sqrt{\lambda-\mu-(f/2-k)^2}\, u} \psi(x), \quad x\in M
\]
be  the solutions of
\[
 \Delta_Z e_{\mu,k,\pm}(\lambda,\psi,(u,x)) = \lambda e_{\mu,k,\pm}(\lambda,\psi,(u,x)).
\]
For $\lambda\in\Complex\setminus[(f/2-k)^2+\mu,\infty)$  the section $e_{\mu,k,+}$ lies in  $L^2\Omega^*(Z,g_Z)$ and $e_{-}$ is not square integrable.
These $e_{\mu,\pm}$ do not satisfy the boundary conditions \eqref{relrbd}. Therefore we consider the linear combination
$
 \text{SIN}_\mu\defgleich\frac{1}{2i}(e_{\mu,+}-e_{\mu,-}).
$
A direct computation shows that 
\begin{multline}\label{tmp01}
  \Big\{ \text{SIN}_{\mu,k,\pm}(\psi),\;du\wedge \text{SIN}_{\mu,k,\pm}(\psi) \mid 0\le k\le f, (\mu,\psi) \text{ so that }\\ \psi \in\Omega^*(B,\H^k(F)),\;  \Delta_{1,0}\psi=\mu\psi \Big\}
\end{multline}
span a complete system of generalized eigenforms of $\Delta_Z$.

The square roots appearing in \eqref{tmp01} and with them $e_\pm$  admit an analytic continuation to $\Sigma_s$.
Let $\psi\in \Omega^{p-k}(B,\H^k(F))$ be an eigenform of $\Delta_{1,0}^{p-k}$ for the eigenvalue $\mu$, $\Lambda\in\Sigma_s$ and $\lambda=\pi_s(\Lambda)$.
Let $\chi$  be a smooth cutoff-function on $X$ with $\chi =1$ on $ [1,\infty)\times M\subset Z$ and $\chi=0$ on $X_0$.
Then $(\Delta_X-\lambda)(\chi e_{\mu,k,-}(\Lambda,\psi))\in \Omega_0^p(X)$ lies in the domain of the analytic continuation $R(\Lambda)$ of the resolvent of $\Delta_X$.

Now the generalized eigenforms of  $\Delta_X$ are defined as
\index{$E_{\mu}(\Lambda,\psi)$}
\begin{equation}
\begin{split}
  E_{\mu}(\Lambda,\psi)(p) &\defgleich \chi(p) e_{\mu,k,-}(\Lambda,\psi,p)-R(\Lambda)(\Delta_X-\lambda)(\chi(p) e_{\mu,k,-}(\Lambda,\psi,p))\label{geneig}
\end{split}
\end{equation}
and an analogous formula for $ E_{\mu}(\Lambda,du\wedge \psi)\in \Omega^{p+1}(X)$.
The generalized eigenform $E_\mu(\Lambda,\psi)$  is uniquely determined by the three properties
\begin{enumerate}[1)]
 \item $E_\mu(\Lambda,\psi) \in \Omega^{p}(X)$ and $E_\mu(\Lambda,\psi)$ is meromorphic in $\Lambda\in\Sigma_s$.
\item $\Delta E_\mu(\Lambda,\psi) =\pi_s(\Lambda) E_\mu(\Lambda,\psi)$ for $\Lambda\in\Sigma_s$
\item $E_\mu(\Lambda,\psi) - \chi e_{-,\mu}(\Lambda,\psi) \in L^2\Omega^{p}(X)$ for $\Lambda\in \FP$.\label{egs3}
\end{enumerate}
\begin{proof}
The proof is the same as in \cite{mu-cusprk1}: Property 1) follows from the corresponding properties of the resolvent and elliptic regularity;  2) and 3) from the definition \eqref{geneig} of $E$.
To show uniqueness, we assume there is a second form $G$ with  properties 1)-3). For $\Lambda\in\FP$ we have
$(E-G)(\Lambda)\in L^2\Omega^p(X)$ because of 3), and from 2)
\[
(\Delta_\text{loc}-\pi_s(\Lambda))(E-G)(\Lambda)=0.
\]
Since $\pi_s(\Lambda)\notin\Reell_+$, this is a contradiction to $\Delta$ being self-adjoint, so that $E=G$.
\end{proof}

By computing the wave operators \eqref{wof},
it can be shown as in \cite{mu-cusprk1} that $\{E_\mu\}$ are a complete system of generalized eigenforms of $\Delta_X P_{ac}$.

We have seen in Proposition \ref{wspek} that the essential spectrum of $\Delta_X$ is determined by fiber-harmonic forms. The fiber-harmonic part 
$\Pi_0(E_\mu|_{Z})$ is called the  \emph{constant term} of $E_\mu$\index{constant term}.

Finally we want to associate  a generalized eigenform to each cohomology class $[\phi]\in H^*(M), [\phi]\neq 0$.
Because of the horizontal Hodge-decomposition Proposition \ref{satzspec} we only need to consider generalized eigenforms for $\mu=0$ and define
\[
E(\Lambda,[\phi])\defgleich E_0(\Lambda,\phi_0)\quad\text{for the unique}\quad\phi_0\in \H^{p-k}(B,\H^k(F))\quad\text{with}\quad [\phi_0]=[\phi].
\]

\subsection{Asymptotics of the constant term}\label{kasym}
We want to determine the asymptotic expansion on the end $Z$ of the constant term $\Pi_0 E(\Lambda,\psi)$ for 
$
\psi\in \H^{p-k}(B,\H^k(F))\subset \H^p(M).
$
As before we define
\[
 e_{\pm,\nu}(\Lambda, \phi) = e^{(\ua\pm i\sqrt{\Lambda-\nu-\ud^2})u} \phi,\qquad \phi\in \Omega^*(B,\H^*(F)),\quad
\Delta_{1,0}\phi=\nu\phi,\quad \Lambda\in\Sigma_s.
\]
We expand in a basis of eigenforms of $\Delta_{1,0}:\Omega^{p-l,l}\to\Omega^{p-l,l}$ for each base-degree
$p-l$. Then the identity  $(\Delta_{Z}-\pi_s(\Lambda))\Pi_0 E(\Lambda,\psi)=0$ gives a system of ordinary differential equations for the coefficients leading to
\begin{multline}\label{asymt:0}
 \Pi_0 E(\Lambda,\psi) = e_{-,0}(\Lambda,\psi)+\sum_{l=0}^f e_{+,0}(\Lambda,\psi_0^{l})
+du\wedge \sum_{l=0}^f e_{+,0}(\Lambda,\hat\psi_0^{l})\\
+\sum_{l=0}^f \sum_{\nu_l>0}e_{+,\nu_l}(\Lambda,\psi_{\nu_l}^{l})
+\sum_{l=0}^f\sum_{\gamma_l>0}du\wedge e_{+,\gamma_l}(\Lambda,\hat\psi_{\gamma_l}^{l}).
\end{multline}
Here $\psi_\nu^l\in \Omega^{p-l,l} ,\;\hat\psi_\gamma^l \in \Omega^{p-l+1,l-1}$ are the eigenforms of $\Delta_{1,0}$,
\[
 \Delta_{1,0}\psi_{\nu_l}^l=\nu_l \psi_{\nu_l}^l,\qquad \Delta_{1,0}\hat\psi_{\gamma_l}^l=\gamma_l \hat\psi_{\gamma_l}^l.
\]
For better readability the indices of the eigenvalues of $\Delta_{1,0}$ will be suppressed in the following.

Let $\mathcal{E}_\nu^{a,b} \subset \Omega^a(B,\H^b(F))$ be the eigenspace of $\Delta_{1,0}$ for the eigenvalue $\nu$.
Now we define linear, in $\Lambda\in\Sigma_s$ meromorphic  operators
\[
S_{0\nu}^{[l]}(\Lambda), \quad T_{0\nu}^{[l]}(\Lambda):\mathcal{E}_0^{*,l}\to\mathcal{E}_\nu^{*,l}
\]
by \eqref{asymt:0} as\index{$T_{0\nu}^{[l]}$}
\[
 T_{0\nu}^{[l]}(\Lambda,\psi)\defgleich\psi_\nu^{l},\qquad  S_{0\gamma}^{[l]}(\Lambda,\psi)\defgleich\hat\psi_\gamma^{l}.
\]
We are particularily interested in the eigenvalue $0$  and the corresponding operator\index{$T_{00}^{[l]}$}
\[
 T_{00}^{[l]}(\Lambda,\cdot): \H^{p}(M)\to\H^{p-l}(B,\H^l(F)),\qquad 0\le l\le f.
\]
Sometimes we will use the abbreviated notation \index{$T_{00}$}$T_{00}=\sum_{l=0}^f T_{00}^{[l]}.$

\begin{satz}\label{sasym8}
The asymptotic expansion of the constant term of $E(\psi,\Lambda)$ with $\psi\in  \H^*(B,\H^k(F))$ is
\begin{multline}\label{aympt1}
 \Pi_0 E(\Lambda,\psi) = e_{-,0}(\Lambda,\psi)+e_{+,0}(\Lambda,T_{00}(\Lambda,\psi))\\
+\sum_{\nu>0}e_{+,\nu}(\Lambda,T_{0\nu}(\Lambda,\psi))
+\sum_{\gamma>0}du\wedge e_{+,\gamma}(\Lambda,S_{0\gamma}(\Lambda,\psi))
\end{multline}
with linear maps $T_{0\nu}(\Lambda), S_{0\nu}(\Lambda): \mathcal{E}_0\to \mathcal{E}_\nu$ which are meromorphic in $\Lambda$.

Similarily
\begin{multline}\label{aympt2}
 \Pi_0 E(\Lambda,du\wedge \psi) = du\wedge e_{-,0}(\Lambda,\psi)+du\wedge e_{+,0}(\Lambda,\check{T}_{00}(\Lambda,\psi))\\
+\sum_{\nu>0}du\wedge e_{+,\nu}(\Lambda,\check{T}_{0\nu}(\Lambda,\psi))
+\sum_{\gamma>0} e_{+,\gamma}(\Lambda,\check{S}_{0\gamma}(\Lambda,\psi))
\end{multline}
\end{satz}

\begin{proof}

It remains to show that in the asymptotic expansion
\begin{multline}\label{asymt:1}
 \Pi_0 E(\Lambda,\psi) = e_{-,0}(\Lambda,\psi)+e_{+,0}(\Lambda,T_{00}(\Lambda,\psi))\\
+du\wedge e_{+,0}(\Lambda,S_{00}(\Lambda,\psi))\\
+\sum_{\nu>0}e_{+,\nu}(\Lambda,T_{0\nu}(\Lambda,\psi))
+\sum_{\gamma>0}du\wedge e_{+,\gamma}(\Lambda,S_{0\gamma}(\Lambda,\psi)).
\end{multline}
the term $du\wedge e_{+,0}(\Lambda,S_{00}(\Lambda,\psi))$ is actually zero.

The two series in the last line  in \eqref{asymt:1} are exponentially decreasing in  $u$, and the term $e_{-,0}(\Lambda,\psi)$ is the only one in \eqref{asymt:1}, which is not square integrable for
$\Lambda\in \FP$.

Similar to \eqref{asymt:1}, $E(\Lambda,du\wedge \psi)$ has the asymptotic expansion
\begin{multline}\label{asymt:du}
 \Pi_0 E(\Lambda,du\wedge \psi) = du\wedge e_{-,0}(\Lambda,\psi)+du\wedge e_{+,0}(\Lambda,\check{T}_{00}(\Lambda,\psi))\\
+ e_{+,0}(\Lambda,\check{S}_{00}(\Lambda,\psi))\\
+\sum_{\nu>0}e_{+,\nu}(\Lambda,\check{S}_{0\nu}(\Lambda,\psi))
+\sum_{\nu>0}du\wedge e_{+,\gamma}(\Lambda,\check{T}_{0\gamma}(\Lambda,\psi))
\end{multline}
with certain linear maps $\check{T}_{0,\nu}(\Lambda), \check{S}_{0,\nu}(\Lambda):\mathcal{E}_0\to \mathcal{E}_\nu$ which are meromorphic in $\Lambda$ as above.

Because $d^{1,0}(T_{00}\psi)=0,$ $d^F(T_{0\nu}\psi)=0,$ and  $d^M(S_{00}\psi)=0$, for $\Lambda\in \FP$
\begin{multline}\label{dzasympt}
 d^Z\Pi_0 E_0(\Lambda,\psi)
=(\ua-i\sqrt{\Lambda-\ud^2}) du\wedge e_{-,0}(\Lambda,\psi) \\
+(\ua+i\sqrt{\Lambda-\ud^2}) du\wedge e_{+,0}(\Lambda,T_{00}(\Lambda,\psi)) \\
 +\sum_{\nu>0}\Big\{(\ua+ i\sqrt{\Lambda-\nu-\ud^2}) du\wedge e_{+,\nu}(\Lambda,T_{0\nu}(\Lambda,\psi ))\\
+ e_{+,\nu}(\Lambda,d^{1,0}(\Lambda,T_{0\nu}(\Lambda,\psi))\Big\}\\
-\sum_{\gamma>0}\ du\wedge d^{1,0}(e_{+,\gamma}(\Lambda,S_{0\gamma}(\Lambda,\psi))).
\end{multline}
 Comparing \eqref{dzasympt} and \eqref{asymt:du} and using the uniqueness of generalized eigenforms gives for $\Lambda\in\Sigma_s$
 \begin{equation} \label{devse1}
  d E(\Lambda, \psi) = (\ua- i\sqrt{\Lambda-\ud^2})(\psi) E(\Lambda, du\wedge \psi).
 \end{equation}
But then because there is no corresponding term in \eqref{dzasympt}, $e_{+,0}(\Lambda,\check{S}_{00}(\Lambda,\psi))$ can not appear in 
\eqref{asymt:du}.
Furthermore
\begin{multline}
 \ast \Pi_0 E(\Lambda,du\wedge\psi) = e_{-,0}(\Lambda,\astm \psi) + e_{+,0}(\Lambda,\astm\check{T}_{00}(\Lambda,\psi)) \\
+\sum_{\nu>0}du\wedge e_{+,\nu}(\Lambda,\astm \check{T}_{0\nu}(\Lambda,\psi))
+\sum_{\nu>0}(-1)^p du\wedge e_{+,\nu}(\Lambda,\astm \check{S}_{0\nu}(\Lambda,\psi)).
\end{multline}
If this is compared  with \eqref{asymt:1} for $\astm\psi$, we first conclude
\[
 E(\Lambda,\astm \psi)= \ast E(\Lambda, du\wedge \psi),
\]
and then,  that the term $du\wedge e_{+,0}(\Lambda,S_{00}(\Lambda,\psi))$ does not appear in \eqref{asymt:1}.
\end{proof}

From this asymptotic expansion we can read off several functional equations for $E$ and the leading coefficient $T_{00}$. 
Let $\{\gamma_{\nu k}\}$ be the generators of  $\text{Aut}(\Sigma_s)$ as described in section  \ref{spektrfl}. Here $\gamma_{\nu k}$ is associated with the ramification point $\nu+d_k^2$.

\begin{cor}\label{funceq}
The generalized eigenforms satisfy the functional equations
\begin{subequations}
\begin{eqnarray} 
  d E(\Lambda, \psi) &=& (\ua- i\sqrt{\Lambda-\ud^2})(\psi) E(\Lambda, du\wedge \psi) \label{devse}\\
 E(\Lambda,\astm \psi)&=& \ast E(\Lambda, du\wedge \psi),\label{emitstern}\\
 (\ua- i\sqrt{\Lambda-\ud^2})(\psi)T_{00}(\Lambda,\astm \psi) &=&\astm (\ua+i\sqrt{\Lambda-\ud^2})( T_{00}(\Lambda,\psi)) \label{Tbez}\\
 E(\Lambda,\psi) &=& E(\gamma_{0k}\Lambda, T_{00}^{[k]}(\Lambda,\psi)) \label{fktgle}
\end{eqnarray}
\end{subequations}
\end{cor}

\begin{proof}

Equations \eqref{devse} and \eqref{emitstern} have been shown in the proof of Proposition \ref{sasym8}.
Comparing the asymptotic expansion in these gives
\begin{equation} \label{fftmat}
(\ua- i\sqrt{\Lambda-\ud^2})(\psi)\check{T}_{00}(\Lambda,\psi) =(\ua+i\sqrt{\Lambda-\ud^2})( T_{00}(\Lambda,\psi)).
\end{equation}
From \eqref{emitstern} and \eqref{fftmat}
\begin{equation*}
  T_{00}(\Lambda,\astm \psi)= \astm \check{T}_{00}(\Lambda,\psi),
\end{equation*}
which leads to the identity \eqref{Tbez}.

Using \eqref{wurzelverh} in the expansion  \eqref{aympt1}, we get
\begin{multline*}
 \Pi_0 E(\gamma_{0k}\Lambda,\psi) = e_{+,0}(\gamma_{0k}\Lambda,\psi)
+e_{-,0}(\Lambda,T_{00}^{[k]}(\gamma_{0k}\Lambda,\psi))
+e_{+,0}(\Lambda,T_{00}^{[\neq k]}(\gamma_{0k}\Lambda,\psi))\\
+\sum_{\nu>0}e_{+,\nu}(\Lambda,T_{0\nu}(\gamma_{0k}\Lambda,\psi))
+\sum_{\eta>0}du\wedge e_{+,\eta}(\Lambda,S_{\eta}(\gamma_{0k}\Lambda,\psi)).
\end{multline*}
The only term on the right hand side, which is not in $L^2$, is $e_{-,0}(\Lambda,T_{00}^{[k]}(\gamma_{0k}\Lambda,\psi))$. 
Uniqueness of the generalized eigenforms gives
\begin{equation*}
 E(\gamma_{0k}\Lambda,\psi) = E(\Lambda, T_{00}^{[k]}(\gamma_{0k}\Lambda,\psi))
\end{equation*}
and the claim follows from $\gamma_{0k}^2=\text{id}$.
\end{proof}

\subsection{Maa{\ss}--Selberg Relations}\label{kmsb}

The Maa{\ss}--Selberg relations will provide us with important information about the position of poles and the asymptotic behavior of generalized eigenforms. The proceeding is  similar to \cite{roelcke}.

For $\psi\in \H^p(M)$ let $\psi^{[k]}$ be the projection of $\psi$ onto $\H^{p-k}(B,\H^k(F))$.

Write $X=X_r\sqcup_M Z_r$ with $Z_r\defgleich[r,\infty)\times M$.
Let $E=E(\phi,\Lambda)=E_{\mu=0}(\phi,\Lambda)$ and define
\[
E^r = E-\Pi_0(E|_{Z_r}).
\]
Because of property \ref{egs3}) of the generalized eigenforms, we have $E^r(\phi,\Lambda)\in L^2\Omega^*(X)$  for $\Lambda\in\FP$.

\begin{satz}\label{maass-selb1} Let $\phi\in \H^*(B,\H^k(F))$.

Let $\nu_1$\index{$\nu_1$} be the smallest positive eigenvalue of $\Delta_{1,0}$ and 
$\tau$ with $0<\tau<\min\{1/4,\nu_1\}$ not a pole of $E(.,\phi)$.

If the fiber-degree of $\phi$ is $k\neq f/2$, the cut-off generalized eigenform $E^r(\tau,\phi)$ has the norm
\begin{multline*}
 \|E^r(\tau,\phi)\|_{L^2(X)}^2 = 
 \frac{e^{2 r \sqrt{d_k^2-\tau}}}{2\sqrt{d_k^2-\tau}} \| \phi \|^2 \\
+r\cdot \sqrt{d_k^2-\tau}\;   \Big\{  \sce{T_{00}(\tau,\phi)}{\phi }-
  \sce{\phi }{T_{00}(\tau,\phi)}\Big\} 
+r\cdot \| T_{00}^{[f/2]}(\tau,\phi)\|^2\\
-i \sqrt{\tau} \Big\{ \Sce{\textstyle{\frac{d}{d\Lambda}|_\tau} T_{00}^{[f/2]}(.,\phi)}{T_{00}^{[f/2]}(\tau,\phi)} -
   \Sce{T_{00}^{[f/2]}(\tau,\phi)} {\textstyle{\frac{d}{d\Lambda}|_\tau} T_{00}^{[f/2]}(.,\phi)}\Big\}\\
-\sum_{\substack{\nu\ge 0; l\\ \nu+d_l^2>0}} \frac{e^{-2 r \sqrt{d_l^2+\nu-\tau}}}{2\sqrt{d_l^2+\nu-\tau}} \| T_{0\nu}^{[l]}(\tau,\phi) \|^2
-\sum_{\gamma> 0; l} \frac{e^{-2 r \sqrt{d_l^2+\gamma-\tau}}}{2\sqrt{d_l^2+\gamma-\tau}} \| S_{0\gamma}^{[l]}(\tau,\phi) \|^2.
\end{multline*}
For $\phi\in \H^*(B,\H^{f/2}(F))$,
\begin{multline*}
 \|E^r(\tau,\phi)\|_{L^2(X)}^2 = 
r (\| \phi\|^2 +  \| T_{00}^{[f/2]}(\tau,\phi) \|^2) \\
-i \sqrt{\tau} \Big\{ \Sce{\textstyle{\frac{d}{d\Lambda}|_\tau} T_{00}^{[f/2]}(.,\phi)}{T_{00}^{[f/2]}(\tau,\phi)} -
   \Sce{T_{00}^{[f/2]}(\tau,\phi)} {\textstyle{\frac{d}{d\Lambda}|_\tau} T_{00}^{[f/2]}(.,\phi)}\Big\}\\
 +\frac{1}{2i\sqrt{\tau}}\Big\{ e^{2 i \sqrt{\tau} r}\sce{T_{00}(\tau,\phi)}{\phi } 
    -e^{-2 i \sqrt{\tau} r}\sce{\phi }{T_{00}(\tau,\phi)} \Big\} \\
-\sum_{\substack{\nu\ge 0; l\\ \nu+d_l>0}} \frac{e^{-2 r \sqrt{d_l^2+\nu-\tau}}}{2\sqrt{d_l^2+\nu-\tau}} \| T_{0\nu}^{[l]}(\tau,\phi) \|^2
-\sum_{\gamma> 0; l} \frac{e^{-2 r \sqrt{d_l^2+\gamma-\tau}}}{2\sqrt{d_l^2+\gamma-\tau}} \| S_{0\gamma}^{[l]}(\tau,\phi) \|^2.
\end{multline*}

\end{satz}
\begin{proof}
Let $\phi\in \H^*(M)$ and $\lambda\in\Complex\smallsetminus \Reell_+$. 
By definition $E^r\upharpoonleft Z_r=\Pi_\perp(E\upharpoonleft Z_r)$, and under the given conditions $\Delta=\Delta_Z$ leaves fiber-harmonic forms on $Z_r$ invariant. Let $M_r=\{r\}\times M$ with the Riemannian metric $\pi^* g^B+e^{-2r}g^F$.
The Green's formula gives
\begin{multline*}
\Scr{\Delta E^r}{E^r}_{L^2(X)}-\Scr{E^r}{\Delta E^r}_{L^2(X)} \\
=\Sce{\Pi_0 E|_{M_r}}{\nabla^Z_{\partial/\partial u} \Pi_0 E|_{M_r}}_{\!L^2(M_r)} 
-\Sce{\nabla^Z_{\partial/\partial u} \Pi_0 E|_{M_r}}{\Pi_0 E|_{M_r}}_{\!L^2(M_r)} 
\end{multline*}
Here we used that $\frac{\partial}{\partial u}$ is outer unit normal vector field to $X_r$ and inner unit normal vector field to $Z_r$.
But $\nabla^Z_{\partial/\partial u}=\frac{\partial}{\partial u}+\kappa$ and
$\tilde\rho_u=e^{-(f/2-\kappa) r}:L^2\Omega^*(M_r,g_u^M)\to L^2\Omega^*(M,g^M)$ is an isometry. Let $g^M=\pi^* g^B+g^F$ the unscaled Riemannian metric on $M$. With $\Sce{}{}_M$ denote the induced $L^2-$norm on $M$. 
Let $\Lambda_1, \Lambda_2\in \FP=\Complex\setminus\Reell^+$, d.h. $\Im \Lambda_1>0, \Im \Lambda_2>0$ and $\lambda_1=\pi_s(\Lambda_1), \lambda_2=\pi_s(\Lambda_2)$. Also let $\lambda_1\neq \overline\lambda_2$,
and $\Lambda_1, \Lambda_2$ should not be poles of $E(\phi)$ and $E(\psi)$ respectively.
\begin{align*}
(\overline\lambda_2-\lambda_1)&\Scr{E^r(\phi,\Lambda_1)}{E^r(\psi,\Lambda_2)}_{L^2(X)} \\
=\;&  \Scr{E^r(\phi,\Lambda_1)}{\Delta E^r(\psi,\Lambda_2)} 
-\Scr{\Delta E^r(\phi,\Lambda_1)}{E^r(\psi,\Lambda_2)}\\
=\;& \Sce{\textstyle\frac{\partial}{\partial u}\bei{r} \tilde\rho_u\Pi_0 E(\phi,\Lambda_1)}{\tilde\rho_u\Pi_0 E(\psi,\Lambda_2)(r,\cdot)}_{\!M}  \\
 &-\Sce{\tilde\rho_u\Pi_0 E(\phi,\Lambda_1)(r,\cdot)}{\textstyle\frac{\partial}{\partial u}\bei{r} \tilde\rho_u\Pi_0 E(\psi,\Lambda_2)}_{\!M} 
\end{align*}

In the sequel we use the notation
\begin{equation} \label{notwu}
 s_\nu^\pm(\Lambda_1,\Lambda_2)=\sqrt{\Lambda_1 -\nu-\ud^2}\pm \sqrt{\overline\Lambda_2 -\nu-\ud^2},\qquad \nu\in\sigma(\Delta_{1,0}).
\end{equation}
Here the branch of the square root, for which $\Im \sqrt{z}>0$ for $z\in \Complex\setminus\Reell$ is chosen;
this implies $\overline{\sqrt{\lambda}}=-\sqrt{\overline{\lambda}}$.

From the asymptotic expansion \eqref{aympt1}
of the constant term $\Pi_0 E(\Lambda,\phi)$ we obtain
\begin{subequations}
\begin{equation} \label{g3}
 \begin{aligned}
(\overline\lambda_2-\lambda_1) & \Scr{E^r(\phi,\Lambda_1)}{E^r(\psi,\Lambda_2)}\\
=\;& -i s_0^- e^{-is_0^+ r }\sce{\phi}{\psi} 
+\sum_{\nu\ge 0} i s_\nu^-  e^{i s_\nu^+ r }  \sce{T_{0\nu}(\Lambda_1,\phi)}{T_{0\nu}(\Lambda_2,\psi)}\\
&+ i s_0^+ \Big\{  e^{i s_0^- r} \Sce{T_{0 0}(\Lambda_1,\phi)}{\psi} 
-e^{-i s_0^-r} \Sce{\phi}{T_{00}(\Lambda_2,\psi)}\Big\}\\
&+\sum_{\gamma> 0} i s_\gamma^-  e^{i s_\gamma^+ r }  \sce{S_{0\gamma}(\Lambda_1,\gamma)}{S_{0\gamma}(\Lambda_2,\psi)}\\
\end{aligned}
\end{equation}
Here  we used the notation
 $\theta(\ua)\sce{\phi}{\psi}\defgleich\sum_l \theta(a_l)\sce{\phi^{[l]}}{\psi^{[l]}}$ for a function $\theta$, where $\psi^{[l]}$ is the projection of $\psi$ onto $\H^{p-l}(B, \H^l(F))$.

Finally because of $s^+ s^-(\Lambda_1,\Lambda_2)=\lambda_1-\overline\lambda_2$
\begin{multline} \label{g4}
\Scr{E^r(\phi,\Lambda_1)}{E^r(\psi,\Lambda_2)}=\\
  \frac{i}{s_0^+}e^{-is_0^+ r }\sce{\phi}{\psi}
- \frac{i}{s_0^-}  \Big\{  e^{i s_0^- r} \Sce{T_{0 0}(\Lambda_1,\phi)}{\psi} 
-e^{-i s_0^-r} \Sce{\phi}{T_{00}(\Lambda_2,\psi)}\Big\}\\
-\sum_{\nu\ge 0} \frac{i}{s_\nu^+}  e^{+i s_\nu^+ r }  \sce{T_{0\nu}(\Lambda_1,\phi)}{T_{0\nu}(\Lambda_2,\psi)}
-\sum_{\gamma> 0} \frac{i}{s_\gamma^+}  e^{+i s_\gamma^+ r }  \sce{S_{0\gamma}(\Lambda_1,\phi)}{S_{0\gamma}(\Lambda_2,\psi)}.
\end{multline}
\end{subequations}

Let $\tau \in \partial_+\FP\simeq \Complex\setminus\Reell^+$, so that $0<\tau<\min\{1/4,\nu_1\}$ and $\eps>0$. In addition we assume that $E(\Lambda,\phi)$
is holomorphic for $\Lambda$ in a neighbourhood of
$\pi_s^{-1}(\tau)$.

The remaining steps are the same as in the proof of Proposition 9.17 in \cite{mu-cusprk1}:
We let $\lambda_1, \lambda_2 \to \tau$ in the upper half plane $\Im \lambda_1>0,$ $\Im \lambda_2>0$.
Then by choice of the square root 
\begin{align*}
s_\nu^-(\lambda_1,\lambda_2)&\to  2\sqrt{\tau-\nu-\ud^2},&
s_\nu^+(\lambda_1,\lambda_2)&\to 0  &&\text{ for } \tau>\nu+\ud^2\\
s_\nu^+(\lambda_1,\lambda_2)&\to 2 i \sqrt{\nu+\ud^2-\tau},& 
 s_\nu^-(\lambda_1,\lambda_2)&\to  0 &&\text{ for } \tau<\nu+\ud^2.
\end{align*}
Only for terms in $\H^{p-f/2}(B,\H^{f/2}(F))$, that is for fiber-degree $f/2$ and $\nu=0$, the inequality $\tau > \nu +\ud^2$ holds. 
We set $\lambda_1=\lambda_2=\tau+i\eps$ in \eqref{g3}.
For $\eps\to 0$, because $\tau$ is no pole:
\begin{multline}\label{awehncv}
\sqrt{\tau}  \| \phi^{[f/2]}\|^2 =  \sqrt{\tau}  \| T_{00}^{[f/2]}\phi \|^2 \\
+i \sqrt{\ud^2-\tau}\;\cdot \;\Big\{   \sce{T_{00}^{[\neq f/2]}\phi}{\phi^{[\neq f/2]} }- \sce{\phi^{[\neq f/2]} }{T_{00}^{[\neq f/2]}\phi}\Big\}
\end{multline}
Then we divide 
\eqref{g3} with the given choice of $\lambda_1, \lambda_2$
by $-2 i \eps$ and employ  \eqref{awehncv} and let $\eps\to 0$ under the assumption that $\tau$ is not a pole of $T_{00}^{[f/2]}$.
Choosing $\phi\in\H(B,\H^{f/2}(F))$ or $\phi\in\H(B,\H^{k}(F))$ with $k\neq f/2$ now gives the claim.
\end{proof}

\subsubsection{Reparametrization}
For the remainder we consider only $0<|\lambda|<\tau_1=\min\{1/4,\nu_1\}$, where $\nu_1$ is the smallest positive eigenvalue of $\Delta_{1,0}$.
The preimage in $\Sigma_s$ under $\pi_s$  of this domain  lies in a double covering of $\Complex$.

For $\lambda\in\Complex\setminus \Reell^+$\index{$s$} let
$
 s=s_{\nu,k}(\lambda)=d_k-i\sqrt{\lambda-\nu-d_k^2},
$
so that $
  \lambda  =s(2d_k-s).
$
In particular in the asymptotic expansion \eqref{aympt1} for $\psi\in  \H^*(B,\H^k(F))$ we choose the parameter
$s=s_{0,k}$, so that
\begin{multline}\label{aymptneu}
 \Pi_0 E(s,\psi) = e^{(a_k-d_k+s) r}\psi+e^{(\ua+\ud-s) r}T_{00}(s,\psi)\\
+\sum_{\nu>0}e^{(\ua+i\sqrt{s(2d_k-s)-\ud^2})r}T_{0\nu}(s,\psi)
+\sum_{\gamma>0}du\wedge e^{(\ua+i\sqrt{s(2d_k-s)-\ud^2})r}S_{0\gamma}(s,\psi).
\end{multline}
Note that $\Re s > d_k$ corresponds to  $\Lambda\in \FP$.
With the notation from \eqref{notwu}
\begin{align*}
 s_{\nu,k}^+(\lambda_1,\lambda_2)=i(s_{\nu,k}(\lambda_1)+\overline{s_{\nu,k}(\lambda_2)}-2d_k),\quad
 s_{\nu,k}^-(\lambda_1,\lambda_2)=i(s_{\nu,k}(\lambda_1)-\overline{s_{\nu,k}(\lambda_2)}).
\end{align*}

Now \eqref{g4}  for $\phi, \psi\in \H^{*}(B,\H^k(F))$, $\hat s\neq \bar s$ and $ \hat s+\bar s\neq 2d_k$ becomes
\begin{multline} \label{g4s}
 (E^r(\hat s,\phi) ,\: E^r(s,\psi))  
=\: \frac{1}{\hat s+\overline s - 2d_k}e^{(\hat s+\overline s - 2d_k) r }\sce{\phi}{\psi} \\
+\frac{1}{\overline s-\hat s} \Big\{  e^{(\overline s-\hat s) r} \Sce{T_{0 0}(\hat s,\phi)}{\psi} 
-e^{(\hat s-\overline s)r} \Sce{\phi}{T_{00}(s,\psi)}\Big\}\\
   -i \sum_{\nu\ge 0; l} \frac{e^{i s_{\nu,l}^+ r }}{ s_{\nu,l}^+}    \sce{T_{0\nu}^{[l]}(\hat s,\phi)}{T_{0\nu}^{[l]}(s,\psi)}
-i \sum_{\gamma> 0; l} \frac{e^{i s_{\gamma,l}^+ r }}{ s_{\gamma,l}^+}    \sce{S_{0\gamma}^{[l]}(\hat s,\phi)}{S_{0\gamma}^{[l]}(s,\psi)}.
\end{multline}

Equations \eqref{devse}, \eqref{emitstern} become
\begin{eqnarray}
 dE(s,\psi)&=&(a_k-d_k+s) E(s, du\wedge \psi) \label{deformel}\\
E(s,\astm\psi)&=&\ast E(s,du\wedge\psi), \label{duformel}
\end{eqnarray}
and \eqref{Tbez} translates to
\begin{multline} 
(a_k-d_k+s) T_{00}(s,\astm\psi)= (a_k+d_k-s) \astm T_{00}^{[k]}(s,\psi)\\
+\sum_{\substack{l=0\\ l\neq k}}^f \big(a_l+i\sqrt{s(2d_k-s)-d_l^2}\,\big) \astm T_{00}^{[l]}(s,\psi).\label{sternswitch}
\end{multline}

\subsection[Poles of $E(s, \phi)$]{Poles of $\mathbf{E(s, \phi)}$}\label{epole}

Let $U$ be a neighbourhood of $0$ and $2d_k$ which is open in  $\Complex$, so that $s(2d_k-s)\in B_{t}(0)$ for  $t<\tau_1=\min\{1/4,\nu_1\}$.
\begin{satz}\label{polposition}
Let $\phi\in \H^{*}(B,\H^k(F))$. Let $\Re(s)\ge d_k$ and $s\in U$, i.e. $s$ lies in the same connected component of $U$ as the point $2d_k$.

For every $\eps>0$ there is a constant $C(\eps)$, so that $\|T_{0\nu}^{[l]}(s,\phi)\|<C(\eps)$ for $s\in U\cap\{\Re s\ge d_k, \Im s\ge \eps\}$. 

In particular poles of $E(s,\phi)$ in $U\cap\{\Re s\ge d_k\}$ must lie in the interval $(d_k,2d_k]$. The order of a pole is 1.
In addition, $T_{00}^{[f/2]}$ is holomorphic at $s=2d_k$.
\end{satz}

\begin{proof}
In \eqref{g4s} set $\hat s=s$, $\sigma=\Re s\ge d_k, \Im s\neq 0$ and $s\in U$. The left hand side of  \eqref{g4s} is non-negative for $\phi=\psi$. Because $\Re(s)\ge d_k$, i.e.
$\lambda=s(2d_k-s)\in \Complex\setminus \Reell^+$, we have
\[
 s_{\nu,l}^+=2 i \:\Im\sqrt{\lambda-\nu-d_l^2}\quad\text{where}\quad\Im\sqrt{\lambda-\nu-d_l^2}\ge 0.
\]
Thus
\begin{align*}
 \sum_{\nu\ge 0; l} &\frac{i e^{i s_{\nu,l}^+ r }}{ s_{\nu,l}^+} \|T_{0\nu}^{[l]}(s,\phi)\|^2 
+ \sum_{\gamma> 0; l} \frac{i e^{i s_{\gamma,l}^+ r }}{ s_{\gamma,l}^+} \|S_{0\gamma}^{[l]}(s,\phi)\|^2 \\
&\le \Big| \frac{e^{( s+\overline s - 2d_k))r }}{ s+\overline s - 2d_k}\Big| \|\phi\|^2 
+2 \Big|\frac{ e^{(\overline s- s) r}}{\overline s- s}\Big| \| T_{0 0}^{[k]}(s,\phi) \| \|\phi\| \nonumber \\
&=  \frac{e^{2(\sigma - d_k)r }}{ 2|\sigma - d_k|} \|\phi\|^2 
+\frac{ 1}{|\Im s|}\| T_{0 0}^{[k]}(s,\phi) \| \|\phi\|.
\end{align*}
Now let $\|\phi\|_{L^2(M,g^M)}=1$. After multiplication with $e^{-2(\sigma - d_k)r}$ we obtain an estimate of the form
\begin{equation} \label{fab1}
 \sum_{\nu\ge 0; l} \frac{i\,C_{\nu,l}}{s_{\nu,l}^+} \|T_{0\nu}^{[l]}(s,\phi)\|^2 
+\sum_{\gamma> 0; l} \frac{i\,C_{\gamma,l}}{s_{\gamma,l}^+} \|S_{0\gamma}^{[l]}(s,\phi)\|^2 
\le  \frac{1}{ 2|\sigma - d_k|}
+\frac{c}{|\Im s|}\| T_{00}^{[k]}(s,\phi) \| \\
\end{equation}
with non-negative numbers $c, C_{\nu,l}$, which depend on $r$, but remain bounded for $r\to\infty$.
Now the asymptotics of both sides of \eqref{fab1} in a neighbourhood of  $s_0$ will be compared.

We distinguish between four cases:
\begin{enumerate}[1)]
\item Let $\Re s_0\ge  d_k, \Im s_0\neq 0$. Then $(s_{\nu,l}^+)^{-1}$ is bounded for $s\to s_0$ and all $(\nu,l)$ and  \eqref{fab1} becomes
\begin{equation} \label{ba247}
 C \|T_{00}^{[k]}\|^2+\sum_{\nu>0} C_\nu \|T_{0\nu}\|^2
+\sum_{\gamma>0} C_\gamma \|S_{0\gamma}\|^2\le \frac{c}{|\Im s|} \|T_{00}^{[k]}\|.
\end{equation}
Let  $n$ be the order of a pole of $T_{00}^{[k]}(.,\phi)$ at $s_0$. Because of  \eqref{ba247} for  $s$ sufficiently close to $s_0$,
\begin{equation}
 \gamma_1 |s-s_0|^{-2n}\le \|T_{00}^{[k]}(s,\phi)\|^2 \le \frac{c}{|\Im s|} \|T_{00}^{[k]}(s,\phi)\| \le \frac{\gamma_2}{|\Im s|}  |s-s_0|^{-n}\label{fab2}
\end{equation}
for certain positive constants $\gamma_1,\gamma_2$.
If $s_0$ is not  real, then $\frac{c}{|\Im s|}$ remains bounded for $s\to s_0$. This implies
\[
 \gamma |s-s_0|^{-2n} \le |s-s_0|^{-n}\qquad \Longrightarrow\quad n\le 0.
\]
Thus $T_{00}^{[k]}(\cdot,\phi)$ is regular in $s_0$, and due to \eqref{ba247} all $T_{0\nu}(\cdot,\phi), S_{0\gamma}(\cdot,\phi)$ must be regular in $s_0$.
\item\label{itemoben} Now let $ \Im s_0=0$ with  $d_k<s_0<2 d_k$, in particular $k\neq f/2$.
In \eqref{fab1} we have to note that possibly
\[
s_{\nu,l}^+(\tau,\tau)=2\sqrt{\tau-\nu-d_l^2}=0
\]
for $\tau\defgleich s_0(2d_k-s_0)$, if $\nu+d_l^2\le d_k^2$.
Therefore we first choose $s_0$ so that $s_{\nu,l}^+(\tau,\tau)\neq 0$. Then \eqref{fab2} shows for $s$ sufficiently close to $s_0$
\[
 \gamma |s-s_0|^{-2n}\le  |s-s_0|^{-n-1},\;\gamma>0 \qquad\Longrightarrow \quad n\le 1,
\]
which means the order of a pole of  $T_{00}^{[k]}(\cdot,\phi)$ at $s_0$ can be at most 1.

If the left hand side of \eqref{ba247} has a pole of order $m$ in $s_0$, as above we conclude $2m\le n+1\le 2$.
So all  $T_{0\nu}$ have a singularity of order less or equal to 1 in $s_0$.

For a given $s_0$ in $(d_k,2d_k)$ choose $(\nu,l)$ so that $s_{\nu,l}^+(\tau,\tau)=0$. Then
\begin{eqnarray*}
 s_{\nu,l}^+ &=& 2 i \:\Im \sqrt{s(2d_k-s)-\nu-d_l^2}\\
&=& 2 i \:\Im  \sqrt{-(s-s_0)^2-2(s_0-d_k)(s-s_0)}
=O(|s-s_0|^{1/2})
\end{eqnarray*}

This means the order $q$ of a pole  of $T_{0\nu}^{[l]}(.,\phi)$ in $s_0$ satisfies $2q+1/2\le 1$, i.e. $q=0$. 
The corresponding terms in the asymptotic expansion are holomorphic. For terms corresponding to other $(\nu,l)$ again the maximal order of a pole in $s_0$ is $1$.
\item\label{itemdavor} Now we consider $s_0=2 d_k$ and $d_k>0$, i.e. $k\neq f/2$. 
\[
s_{0,f/2}^+=\sqrt{s(2d_k-s)}+\sqrt{\bar s(2d_k-\bar s)}=O(|s-2d_k|^{1/2})\quad\text{for}\quad s\to s_0.
\]
The term $(s_{\nu,l}^+)^{-1}$ remains bounded in the limit $s\to s_0$ for all other $(\nu,l)$.
Thus \eqref{fab1} takes the form
\begin{multline*}
 \gamma_1 |s-2d_k|^{-1/2}\|T_{00}^{[f/2]}(s,\phi)\|^2+ C \|T_{00}^{[k]}(s,\phi)\|^2\\
+\sum_{\nu>0} C_\nu \|T_{0\nu}(s,\phi)\|^2+\sum_{\gamma>0} C_\gamma \|S_{0\gamma}(s,\phi)\|^2
\le \frac{c}{|\Im s|} \|T_{00}^{[k]}(s,\phi)\|
\end{multline*}
Analoguous to \ref{itemoben}) we conclude that $T_{00}^{[k]}$ and $T_{0\nu}^{[l]}$  can have a pole of order at most one in $s_0$ and that
 $T_{00}^{[f/2]}$ is regular in $s_0$.
\item Finally let $s_0=0$ for $k=f/2$. 
Then $|s_{0,f/2}^+|<\gamma |s|$ for $s$ near $0$.
Here we have to consider the inequality
\begin{multline*}
 \gamma_1 |s|^{-1}\|T_{00}^{[f/2]}(s,\phi)\|^2+\sum_{\nu>0} C_\nu \|T_{0\nu}(s,\phi)\|^2
+\sum_{\gamma>0} C_\gamma \|S_{0\gamma}(s,\phi)\|^2 \\
\le \frac{c}{|\Im s|} \|T_{00}^{[f/2]}(s,\phi)\|+\frac{c'}{|s|}.
\end{multline*}
Again this implies that $T_{0\nu}$ and $S_{0\gamma}$ must be regular in $s_0$.
\end{enumerate}
\end{proof}

\begin{satz}\label{satzpolord}
The order of a pole of $E(\cdot ,\phi)$ in $s_0\in U$ is the maximum of the orders of poles of $T_{0\nu}^{[l]}(\cdot ,\phi)$ in $s_0$.
In particular $E(\cdot ,\phi)$ has a pole of order at most 1 in $2d_k$.
\end{satz}
\begin{proof}
This follows directly from \eqref{g4s} and the proof of Proposition \ref{polposition}.
\end{proof}

\begin{lem}\label{lemma3}
Let $\bar s$ be not a pole of  $E^r(.,\phi)$ and $s$  not a pole of $E^r(.,\psi)$.
Then
\[
 \Sce{T_{00}(\bar s,\phi)}{\psi}= \Sce{\phi}{T_{00}(s,\psi)}.
\]
\end{lem}
\begin{proof}
Under the assumption that neither $\bar s$ is a pole of $E^r(.,\phi)$ nor $s$ is  a pole of $E^r(.,\psi)$, from \eqref{g3} (or \eqref{g4s}) we see for $\hat s=\bar s$
\[
 0= 2(\bar s-2d_k) \{  \Sce{T_{00}(\bar s,\phi)}{\psi} - \Sce{\phi}{T_{00}(s,\psi)}\}.
\]
Because of Proposition \ref{satzpolord} the $T_{00}$ are holomorphic in $\bar s$ and $s$, and the claim follows.
\end{proof}

\subsection{Residues}
We have seen that poles in a neighbourhood $U$ of $2d_k$ can lie only in $U\cap(d_k,2d_k]$, and can only have order one. Let $s_0\in U\cap(d_k,2d_k]$ and  $\tau= s_0(2d_k-s_0)$.

First we choose $s_0$ such that $s_{\nu,l}^+(\tau,\tau)\neq 0$ for all $l, \nu$.
After multiplication of \eqref{g4s} with $(\hat s-s_0)(\overline s-s_0)$ we first let $s\to s_0$ and then $\hat s\to s_0$.
\begin{align}
&(\res_{s_0} E^r(\cdot,\phi),\res_{s_0} E^r(\cdot,\psi))\nonumber\\
&= \Sce{\phi}{\res_{s_0} T_{00}(\cdot,\psi)}
  -\sum_{\nu\ge 0; l} \frac{e^{-2r \sqrt{\nu+d_l^2-\tau} }}{ 2\sqrt{\nu+d_l^2-\tau}}\sce{\res_{s_0} T_{0\nu}^{[l]}(\hat s,\phi)}{\res_{s_0} T_{0\nu}^{[l]}(s,\psi)} \nonumber\\
&\qquad-\sum_{\gamma> 0; l} \frac{e^{-2r \sqrt{\gamma+d_l^2-\tau} }}{ 2\sqrt{\gamma+d_l^2-\tau}}\sce{\res_{s_0} S_{0\gamma}^{[l]}(\hat s,\phi)}{\res_{s_0} S_{0\gamma}^{[l]}(s,\psi)}
\end{align}
Both sums converge to $0$ for $r\to \infty$.

Now let $s_0=2d_k$, that is $\tau=0$. Then we have
\[
s_{\nu,l}^+(\tau,\tau)=0\iff (\nu=0 \land  l=f/2).
\]
First we consider $d_k>0$.
\begin{align*}
  & \lim_{\hat s\to 2 d_k} \lim_{s\to 2d_k}\frac{e^{i s_{0,f/2}^+(\hat s, s) r }}{ s_{0,f/2}^+(\hat s, s)}\sce{(\hat s-2d_k)T_{00}^{[f/2]}(\hat s,\phi)}{(s-2d_k) T_{00}^{[f/2]}(s,\psi)}\\
=& \lim_{\hat s\to 2 d_k}\frac{e^{i \sqrt{\hat s(2d_k-\hat s)} r }}{ \sqrt{\hat s(2d_k-\hat s)}}\sce{(\hat s-2d_k)T_{00}^{[f/2]}(\hat s,\phi)}{\res_{2d_k} T_{00}^{[f/2]}(\psi)}=0,
\end{align*}
where in the last line we used that $T_{00}^{[f/2]}$ is holomorphic at $2d_k$.

Analogous for $d_k=0$, i.e. $\psi\in \H^{*}(B,\H^{f/2}(F))$:
\begin{align*}
  & \lim_{\hat s\to 0} \lim_{s\to0}\frac{e^{i s_{0,f/2}^+(\hat s, s) r }}{ s_{0,f/2}^+(\hat s, s)}\sce{\hat s T_{00}^{[f/2]}(\hat s,\phi)}{s T_{00}^{[f/2]}(s,\psi)}\\
=& \lim_{\hat s\to 0}\frac{e^{\hat s r }}{\hat s}\sce{\hat s T_{00}^{[f/2]}(\hat s,\phi)}{\res_{0} T_{00}^{[f/2]}(\psi)}=0.
\end{align*}
By interchanging the limits we  arrive at
\begin{satz}\label{residuen}
Let $\H^{*}(B,\H^{k}(F))$. Let $s_0=2d_k$ or $s_0\in U \cap (d_k, 2 d_k)$, so that $s_{\nu,l}^+(\tau,\tau)\neq 0$ for all $l, \nu$. Then
 $\res_{s_0} E(\cdot,\phi) \in L^2\Omega^*(X)$ and
 \[
 \Scr{\res_{s_0} E(\cdot,\phi)}{\res_{s_0} E(\cdot,\psi)}
= \Sce{\phi}{\res_{s_0} T_{00}(\cdot,\psi)}= \Sce{\res_{s_0} T_{00}(\cdot,\phi)}{\psi}.
 \]
\end{satz}

The residues at $2d_k$ are of special interest. From now on we will use the notation\index{$\tC(\phi)$}\index{$\tE(\phi)$} 
\[
\tC(\phi)\defgleich\res_{2d_k} T_{00}(.,\phi),\quad
{\tE}(\phi)\defgleich\res_{2d_k} E(.,\phi)
\]
and 
$\H^{*,k}(M)\defgleich \H^{*}(B,\H^{k}(F)).$

An immediate consequence of Proposition \ref{residuen} is
\begin{cor}\label{zerlres}
\begin{enumerate}[a)]
 \item  ${\tC}^{[k]}=({\tC}^{[k]})^*, \quad {\tC}^{[k]}\ge 0$ and $\H^{*,k}(M)$ splits into the orthogonal direct sum
\[
 \H^{*,k}(M)=\ker {\tC}^{[k]}\oplus \bild {\tC}^{[k]}.
\]
\item For $\phi\in\H^{*,k}(M)$,
\[
E(s,\phi) \text{ is holomorphic at } s=2d_k \iff \phi \in \ker {\tC}^{[k]} \iff \phi \in \ker {\tC}
\]
\item For $\phi\in \H^{*,k}(M)$ the residues $\tE(\phi)$ and $\widetilde  E(du\wedge \phi)$ are in $L^2\Omega^*(X)$.
\end{enumerate}
\end{cor}
\begin{proof}
a) is immediate from Proposition \ref{residuen}. For b), let $\phi \in \ker {\tC}^{[k]}$. Proposition \ref{residuen} shows
\[
\| {\tE}(\phi) \|^2=\sce{\widetilde{C}^{[k]}(\phi)}{\phi}=0 \Rightarrow {\tE}(\phi)=0.
\]
If conversely $ {\tE}(\phi)=0$, from Proposition \ref{satzpolord} we conclude that  $T_{00}(.,\phi)$ is holomorphic at $2d_k$.
Finally c)  follows from Proposition \ref{residuen} and  $E(s,\astm\psi)=\ast E(s,du\wedge\psi)$.
\end{proof}

%
%
%
%
%

In particular the residues are closed, because
\[
 \|d \tE\|^2+ \|\delta \tE\|^2=\scr{\Delta \tE}{\tE} =0.
\]
This gives a map $\tE\mapsto [\tE]$ from a subspace of $\H_{(2)}^p(X)$ to $H^p(X)$.

\section{A Hodge--type Theorem}

\subsection{Harmonic representatives}
Using generalized Eigenforms, to each class in $H^p(M)$ a harmonic representative of a class in  $H^p(X)$ will be associated.
Let $\phi\in \H^{*,k}(M)$ be a representative of a class in $H^p(M)$.

First we consider the case $k\neq \frac{f}{2}$.
\begin{satz}\label{satz47}
Let $k>f/2$ and $\phi\in \H^{*,k}(M)$.
Then $E(\cdot,\phi)$ is holomorphic at $s=2d_k$.
$E(2d_k,\phi)$ is closed if and only if ${\tC}^{[f-k]}(\astm\phi)=0$.

Let $k<f/2$ and $\phi\in \H^{*,k}(M)$. Then  $E(\cdot,\phi)$ is holomorphic at $s=2d_k$ if and only if $\phi\in\ker {\tC}^{[k]}$.
\end{satz}
\begin{proof}
Let $\phi\in \H^{*,k}(M)$.
For $k>f/2$, that is $a_k<0$,   in \eqref{sternswitch} we consider only the contribution coming from fiber degree $f-k$:
\begin{equation}
(2d_k-s) T_{00}^{[f-k]}(s,\astm\phi) =  s \astm T_{00}^{[k]}(s,\phi). \label{reszue}
\end{equation}
So if $T_{00}^{[f-k]}(s,\astm\phi)$ at $2d_k$  has a singularity of order $n$, then 
$T_{00}^{[k]}(s,\phi)$ has a singularity of order $n-1$ there. But the order of a pole at $2d_k$ is at most $1$ because of Proposition \ref{satzpolord}, so that
$T_{00}^{[k]}(s,\psi)$ must be holomorphic  at $s=2d_k$.
From Corollary \ref{zerlres} we conclude that $E(s,\phi)$ is holomorphic at $2d_k$. Equations \eqref{deformel} and \eqref{duformel} give
 \begin{align*}
  d E(s,\phi) &= (a_k-d_k+s) E(s,du\wedge \phi)\\
 &= (-2d_k+s) E(s,du\wedge \phi) = \pm(-2d_k+s) \ast E(s,\astm\phi) .
 \end{align*}
Thus $E(2d_k,\phi)$ is closed if and only if  $E(.,\astm\phi) $ is holomorphic at $2d_k$. From Corollary
\ref{zerlres} we know
$
 \H^{*,f-k}(M)=\ker {\tC_{2d_k}}^{f-k}\oplus \bild{\tC_{2d_k}}^{f-k}
$
and the remaining statements follow.
\end{proof}

\begin{satz}\label{satz48}
 $E(s,du\wedge \phi)$ is holomorphic and exact in $s=2d_k$ for  $k<f/2$.
 $\widetilde{E}(du\wedge \phi)=dE(2d_k,\phi)$ for $k>f/2$.

In particular both $E(2d_k, du\wedge\phi)$ (if existent) and $ \widetilde{E}(du\wedge \phi)$ are zero in the de~Rham cohomology of $X$ for all $\phi\in \H^{*,k}(M)$ with $k\neq f/2$.
\end{satz}
\begin{proof}
Let $k<f/2$ and $\phi\in  \H^{*,k}(M)$. In equation \eqref{duformel}
the left hand side is holomorphic at $s=2d_k$.
For the residues follows
$
 \widetilde{E}(du\wedge \phi)=0,
$ so that $E(.,du\wedge \phi)$ is holomorphic in $2d_k$.

For the  exactness we use \eqref{deformel},
$
 E(2d_k, du\wedge\phi)=\frac{1}{2d_k} d E(2d_k,\phi).
$
If otherwise $\phi\in  \H^{*,k}(M)$ for $k>f/2$,
then again from  \eqref{deformel}
\[
 d E(s,\phi)=(s-2d_k) E(s, du\wedge\phi),
\]
and so the claim about the residue follows in the limit $s\to 2d_k$.
\end{proof}

Next we consider the middle fiber degree, $\phi\in \H^{*,f/2}(M)$.
From Proposition \ref{polposition} we know that $T_{00}^{[f/2]}(s,\phi)$ is holomorphic in $s=0$.
The functional equation \eqref{fktgle}
here implies
\[
 T_{00}^{[f/2]}(0, T_{00}^{[f/2]}(0,\phi))=\phi\in \H^{*,f/2}(M).
\]
Additionally,  $(T_{00}^{[f/2]})^*=T_{00}^{[f/2]}$ by Lemma \ref{lemma3}.
Thus for $p\ge f/2$ there is a decomposition of $\H^{*,f/2}(M)$ into the direct orthogonal sum
\begin{equation} \label{splitf2}
\H^{p-f/2,f/2}(M)=\H_+^p\oplus \H_-^p
\end{equation}
with\index{$\H_\pm^p$}
\[
 \H_\pm^p=\{\phi\in \H^{p-f/2,f/2}(M)\mid T_{00}^{[f/2]}(0,\phi)=\pm\phi\}.
\]
The Hodge-isomorphism $\H^p(M)\to H^p(M)$ gives a corresponding splitting 
\[
H^{p-f/2}(B,\H^{f/2})=:H^{p-f/2,f/2}(M)=\H_+^p(M)\oplus \H_-^p(M).
\]

\begin{lem}\label{hodgeiso}
The Hodge-star-operator
\[
 \astm: \H_\pm^{p} \to \H_\mp^{n-p}
\]
is an isomorphism.
\end{lem}
\begin{proof}
This follows from the functional equation \eqref{sternswitch},
\[
 T_{00}^{[f/2]}(s,\astm\phi)=-\astm  T_{00}^{[f/2]}(s,\phi).\qedhere
\]
\end{proof}

\begin{satz}\label{satz49}
Let $\phi\in \H^{*,f/2}(M)$.
The generalized eigenform  $E(s,\phi)$ is holomorphic and closed in $s=0$. $E(s, du\wedge\phi)$ is holomorphic in $0$.
\end{satz}
\begin{proof}
That $E$ is holomorphic in $0$ follows from Propositions \ref{polposition} and \ref{residuen}.
Then  $E(s,du\wedge \phi)$ is holomorphic in $s=0$ because of \eqref{duformel} and \eqref{deformel}
implies $d E(s,\phi)=0$.
\end{proof}

\begin{satz}\label{satz50}
$E(0,\phi)=0$ for $\phi\in \H_-$,  and $E(0,du\wedge\phi)=0$ for $\phi\in \H_+$.
\end{satz}
\begin{proof}
From Proposition \ref{maass-selb1} for real $\tau$ near $0$:
\begin{multline}
 \|E^r(\tau,\phi)\|^2 = 
r (\| \phi\|^2 +  \| T_{00}^{[f/2]}(\tau,\phi) \|^2) \\
-i \sqrt{\tau} \Big\{ \Sce{\textstyle{\frac{d}{d\Lambda}|_\tau} T_{00}^{[f/2]}(.,\phi)}{T_{00}^{[f/2]}(\tau,\phi)} -
   \Sce{T_{00}^{[f/2]}(\tau,\phi)} {\textstyle{\frac{d}{d\Lambda}|_\tau} T_{00}^{[f/2]}(.,\phi)}\Big\}\\
 +\frac{1}{2i\sqrt{\tau}}\Big\{ e^{2 i \sqrt{\tau} r}\sce{T_{00}(\tau,\phi)}{\phi } 
    -e^{-2 i \sqrt{\tau} r}\sce{\phi }{T_{00}(\tau,\phi)} \Big\} \\
-\frac{e^{-2  i r \sqrt{\tau}}}{2 i\sqrt{\tau}} \| T_{00}^{[f/2]}(\tau,\phi) \|^2
+G_r
\end{multline}
where  the remainder term $G_r\xrightarrow[r\to\infty]{}0$.
Since $T_{00}^{[f/2]}(\tau, \phi)$ is regular in $\tau=0$ and $\|\phi\|^2=\|T_{00}^{[f/2]}(0,\phi)\|^2$, we conclude in the limit $\tau\to 0$
 \begin{equation*}
\|E^r(0, \phi)\|^2 = 
2 r\cdot \: \| \phi\|^2 
+r \cdot\Big\{ \sce{T_{00}^{[f/2]}(0, \phi)}{\phi} +\sce{\phi}{T_{00}^{[f/2]}(0, \phi)} \Big\} 
+G_r.
\end{equation*}
Now if $\phi\in \H_-$, this gives the equality
$
\|E^r(0, \phi)\|^2 = G_r,
$
and the first claim follows by taking the limit $r\to\infty$.

Now let $\phi\in \H_+ \Longrightarrow \astm\phi\in \H_-$. The second statement then is a consequence of
\[
  E(s,du\wedge \phi)=\pm\ast E(s,\astm\phi)=0.\qedhere
\]
\end{proof}

Now we want to collect all information about harmonic representatives defined by generalized eigenforms in a single map.
Let $\phi\in \H^{p-k}(B,\H^k(F))$.
We define a map $\Xi: \H^p(M)\to \Omega^p(X)$ by
\begin{equation*}
 \Xi(\phi)=
\begin{cases}
\tE(\phi),&k<f/2\quad\text{and}\quad\phi\in \bild \tC^{[k]}\\
E(0,\phi),&k=f/2\quad\text{and}\quad\phi\in \H_+\\
E(2d_k,\phi),&k>f/2\quad\text{and}\quad\astm\phi\in \ker \tC^{[f-k]}\\
0,&\text{otherwise}
\end{cases}
\end{equation*}
and linear extension to $\H^p(M)$. The differential forms $\Xi(\phi)$ will be called \emph{singular values}.

From  Corollary \ref{zerlres}, Propositions \ref{satz47} and \ref{satz49} it follows that singular values are closed harmonic differential forms; as such they represent classes in $H^p(X)$. Thus $\Xi$ extends to a map  $\Xi:H^p(M)\to H^p(X)$ by setting
\[
 \Xi([\phi])= [\Xi(\phi)],\qquad \phi\in \H^p(M).
\]

\subsection{Restriction map}
We want to consider the restriction of classes of singular values to the ``boundary'' $M$. More precisely, we identify $Y_s=\{s\}\times M\subset Z$ for $s>0$ with $M$ and 
let $r=i_s^*:H^p(X)\to H^p(M)$ with the inclusion $i_s:Y_s\hookrightarrow X$. To see that $r$ is well-defined, let $\gamma$ be a cycle in  $M$ and let $\gamma_s, \gamma_t$ denote the corresponding cycles in $Y_s$ and $Y_t$, respectively. From Stokes' theorem  for $[\theta]\in H^p(X)$
\begin{equation} \label{zyka}
 \int_{\gamma_s} i_s^*\theta- \int_{\gamma_t}i_t^*\theta= \int_{\partial([s,t]\times \gamma)} \theta= \int_{[s,t]\times \gamma} d\theta =0.
\end{equation}
The Theorem of de~Rham  states that the map $\Psi^*:H^p(M)\to H^p_{\text{sing}}(M)$ induced by $\Psi(\theta)(\gamma)=\int_\gamma\theta$ is an isomorphism, so that $[i_t^*\theta]=[i_s^*\theta]$.

\begin{lem}\label{randrham}
Let $\theta\in \Omega^*(X)$ be a closed form such that the restriction $\theta|_Z=\theta(u,y)$  is rapidly decreasing for $u\to\infty$.
Then $r([\theta])=0$.

Furthermore if $\theta\in\Omega^p(X)$ is a closed form with $\Pi_0 (\theta|_Z)=0$, then $r[\theta]=0$.
\end{lem}
\begin{proof}
Since $i_t^*\theta$ is rapidly decreasing for $t\to\infty$,  \eqref{zyka} shows
\[
 \int_{\gamma} i^*\theta=0,
\]
which implies $r[\theta]=0\in H^p(M).$

Now let $\Pi_0 (\theta|_Z)=0$. In particular there is a $u>0$ such that
\[
 i_u^*\theta \perp \Omega^*(B,\H^*(F))\quad \Longrightarrow\quad i_u^*\theta \perp \H^*(M).
\]
This shows $[i_u^*\theta]=0\in H^p(M)$, and we have seen that this class is independent of $u$.
\end{proof}

\begin{cor}
Let $\phi\in \H^{p-k}(B,\H^k(F))$. Then the restriction $(r\circ\Xi)[\phi]$ of singular values  to $H^p(M)$ is given by
\begin{subequations}
\[
(r\circ\Xi)[\phi]=\left\{
\text{
\begin{minipage}{0.8\textwidth}
\begin{flalign}
&[\sum_{l<f/2}{\tC}^{[l]}(\phi)] ,&&\text{for }k<f/2,\quad \phi\in \bild \tC^{[k]},  \label{einsch1}\\ 
& 2[\phi]+\sum_{l<f/2}[T_{00}^{[l]}(2d_k,\phi)] ,&& \text{for }k=f/2, \quad\phi\in \H_+\label{einsch2} \\
& [\phi]+\sum_{l\le f/2}[T_{00}^{[l]}(2d_k,\phi)] ,&&\text{for } k>f/2,\quad \astm\phi\in \ker \tC^{[f-k]}\label{einsch3}\\
&0&&\text{otherwise}\nonumber%
\end{flalign}\nonumber
\end{minipage}
}
\right.
\]
\end{subequations}
\end{cor}

\begin{proof}
From the second statement of Lemma \ref{randrham} we obtain $r[E]=r[\Pi_0 E]$ if $E$ is closed, and  $r[\tE]=r[\Pi_0 \tE]$. Thus it is sufficient to examine the restrictions of the respective constant terms to the ``boundary'' $M$.

From the asymptotic expansion \eqref{aymptneu} of $\Pi_0 E$ it follows for $k>f/2$
\[
 \Pi_0 E(2d_k,\phi)=\phi + \sum_{l\le f/2} T_{00}^{[l]}(2d_k,\phi) + \sum_{l> f/2} e^{-2d_l u}T_{00}^{[l]}(2d_k,\phi)+\theta
\]
with exponentially decreasing $\theta$.
Then Lemma \ref{randrham} gives
\[
 r[\Xi[\phi]]=r[E(2d_k,\phi)]=r[\Pi_0 E(2d_k,\phi)]= [\phi]+\sum_{l\le f/2}[T_{00}^{[l]}(2d_k,\phi)].
\]
Similarily we derive  \eqref{einsch2} for $r[E(0,\phi)]$ in the case $k=f/2$ and $\phi\in \H_+$.

Finally in the case $k<f/2$,
\[
  \Pi_0 \tE(2d_k,\phi)=\sum_{l<f/2} {\tC}^{[l]}(\phi) + \sum_{l> f/2} e^{-2d_l u}{\tC}^{[l]}(\phi)+\tilde\theta,
\]
where we used that $T_{00}^{[f/2]}$ is holomorphic (Proposition \ref{polposition}).
This gives  \eqref{einsch1}.
\end{proof}

\subsection{A Hodge-type theorem}\label{ktheorem}
Let $X$ be a manifold with fibered cusp metric, such that conditions (A) and (B) hold.
Let
\[
\mathcal{A}^p=\bigoplus_{k=0}^f \mathcal{A}^{(p-k,k)}\qquad\text{with}\qquad \mathcal{A}^{(p-k,k)}\defgleich
\begin{cases}
\bild {\tC}^{[k]}, & k<f/2\\
\H_+^p,&k=f/2\\
\astm\ker {\tC}^{[f-k]}, & f/2<k\le \min\{f,p\}\\
0,& k> \min\{f,p\}
\end{cases}
\]

\begin{xthm}\label{thm2}
Let $\mathfrak{h}:\H^p(M)\to H^p(M), \phi\mapsto [\phi]$ be the Hodge-isomorphism. Then
$\bild (r: H^p(X)\to H^p(M))=\mathfrak{h}(\mathcal{A}^p)$.
\end{xthm}
\begin{proof}
Let $R^p\defgleich\bild (r: H^p(X)\to H^p(M))$ and $R^{(j,k)}\defgleich\bild\big(r^{(j,k)}: H^{j+k}(X)\to H^{(j,k)}(M)\big)$ be the image of the restriction map $r$, projected onto $H^{(j,k)}(M)\defgleich H^j(B,\H^k(F))$.
From Stokes' theorem and the Hodge decomposition on the closed manifold $M$ it follows that 
\begin{equation}
 q([\phi], [\psi])=\int_M \phi\wedge\psi= (\astm \phi, \psi)_{L^2\Omega^{p}(M)}\label{qformel}
\end{equation}
defines a non-degenerate bilinear form  $q: H^{(a,k)}(M)\times H^{(n-a-f,f-k)}(M)$ such that
$R^{(a,k)}$  and  $R^{(n-f-a,f-k)}$ are $q-$orthogonal.

We claim
\begin{equation} \label{step3}
R^{(*,k)}=\mathfrak{h}(\mathcal{A}^{(*,k)}).
\end{equation}

To prove \eqref{step3}, first let $k<f/2$.
From \eqref{einsch1} we get $\mathfrak{h}(\bild\tC^{[k]})\subset R^{(j,k)}$, so that
\[
 \bild\tC^{[k]}\subset \mathfrak{h}^{-1}(R^{(j,k)}) \subset \H^{j,k}(M)=\bild\tC^{[k]}\oplus \ker\tC^{[k]}
\]
by Corollary \ref{zerlres}. To show
\begin{equation}
 \mathfrak{h}(\bild\tC^{[k]})=R^{(j,k)},\qquad k<f/2,\label{pres2}
\end{equation}
it is sufficient to prove $\mathfrak{h}^{-1}(R^{(j,k)})\cap \ker\tC^{[k]}=\{0\}$.
For $v\in \ker\tC^{[k]}$, from \eqref{einsch3}
\[
 R^{(n-f-j,f-k)}\ni (r^{n-f-j,f-k}\circ\Xi)([\astm v])=[\astm v],\qquad k<f/2.
\]
This shows 
\begin{equation}
\mathfrak{h}(\astm  \ker\tC^{[k]})\subset  R^{(\cdot,f-k)},\qquad k<f/2.\label{pres1a}
\end{equation}
In particular if $[v]\in R^{(j,k)}$ with $v\in \ker\tC^{[k]}$, then $\|v\|^2=q([\astm v],[v])=0$  and
 \eqref{pres2} follows.

Now let $v\in \H^{*,f-k}(M)$ with  $[v]\in  R^{(\cdot,f-k)}$. From \eqref{pres2} and \eqref{qformel}
\[
 [v]\in (R^{(\cdot,k)})^{\perp_q}=\mathfrak{h}((\bild\tC^{[k]})^{\perp_q}) \Longrightarrow \astm v\in (\bild\tC^{[k]})^\perp=\ker\tC^{[k]},
\]
so that
\begin{equation}
 R^{(\cdot,f-k)}\subset \mathfrak{h}(\astm \ker{\tC}^{[k]}),\qquad k<f/2.\label{pres1b}
\end{equation}

Now let $k=f/2$.  From \eqref{einsch2} and \eqref{splitf2}
\begin{equation*}
H_+^{(j,f/2)}\subset R^{(j,f/2)} \subset H^{(j,f/2)}=H_+^{(j,f/2)}\oplus H_-^{(j,f/2)}
\end{equation*}
But $\astm: \H_-^{(j,f/2)}\to \H_+^{(n-j-f,f/2)}$ is an isomorphism due to Lemma \ref{hodgeiso}.
Then
\[
[v]\in H_-^{(j,f/2)}\cap R^{(j,f/2)} \Rightarrow [\astm v]\in H_+^{(n-f-j,f/2)}\subset  R^{(n-f-j,f/2)} \subset (R^{(j,f/2)})^{\perp_q}
\]
Again $[v]=0$ and thus
\begin{equation}
 H_+^{(j,f/2)}(M)=R^{(j,f/2)}.\label{pres3}
\end{equation}
Equations \eqref{pres1a}, \eqref{pres1b}, \eqref{pres2} and \eqref{pres3} prove \eqref{step3} and the theorem.
\end{proof}

An example is given by the case of a manifold with cylindrical end, that is the metric on $Z$ is $du^2+g^M$. Then the fibers are points, and Theorem \ref{thm2} shows
\[
 \im\big(H^p(X)\to H^p(M)\big)=\H^p_+.
\]
This is Theorem 3.1 in \cite{MuStro}.

The main result of this article is
\begin{xthm}\label{thm1}

Let $H_!^p(X)\defgleich\bild (H_c^p(X)\to H^p(X))$ be the image of cohomology with compact support in  the de~Rham-cohomology.
Let $H_{\mathrm{inf}}^p(X)$ be a complementary space to $H_!^p(X)$ in $H^p(X)$,\index{$H_{\rm inf}^p(X)$}\index{$H_{\shriek}^p(X)$}
\[
H^p(X)=H_!^p(X)\oplus H_{\mathrm{inf}}^p(X).
\]
Let $R^p\defgleich\bild (r: H^p(X)\to H^p(M))$.
Then $\Xi(R^p)$ is isomorphic to $H_{\text{\upshape{inf}}}^p(X)$
and\/ $\Xi(H^p(M))=\Xi(R^p).$
\end{xthm}

\begin{proof}
By the definition of relative de~Rham-cohomology we have 
\[
\ker (r: H^p(X)\to H^p(M))\cong\bild (H^p(X,M)\to H^p(X))
\]
 and it is well known that $H_c^p(X)$  is isomorphic to $H^p(X,M)$. 
This implies that 
\begin{equation}
r: H_{\text{inf}}^p(X) \to \bild (r: H^p(X)\to H^p(M))
\end{equation}
is an isomorphism.

We  show that  $r\circ\Xi : R^p \to R^p $ is an isomorphism, then
\[
 (r|_{\bild r})^{-1}\circ r :\Xi(R^p)\xrightarrow[]{\simeq} H_{\text{inf}}^p(X)
\]
gives the desired isomorphism.

First, from equations \eqref{einsch1}, \eqref{einsch2} and \eqref{einsch3} we obtain
\[
 (r^{(p-k,k)}\circ \Xi)(\mathcal{A}^{(p-k,k)})\subset \mathcal{A}^{(p-k,k)},
\]
and altogether $r\circ \Xi(\mathcal{A})\subset \mathcal{A}$.

Next we show that $r\circ \Xi$ is injective on $\mathcal{A}$.
Let $\phi\in \H^{p-k,k}(M)$ and $r(\Xi([\phi]))=0$. For $k\ge f/2$, from  \eqref{einsch2} and \eqref{einsch3} it follows that $[\phi]=0$. 
If otherwise $k<f/2$ and $\phi\in \bild{{\tC}^{[k]}}$, then by \eqref{einsch1}
\[
 r(\Xi([\phi]))=0 \Longrightarrow {\tC}^{[k]}(\phi)=0.
\]
Then Corollary \ref{zerlres} shows  $\phi=0$.

%

From Theorem \ref{thm2} we get $R^p=\mathcal{A}$ and hence $r\circ\Xi:R^p\to R^p$ is an isomorphism. Furthermore this shows the second claim of the theorem, as
$
 \Xi(H^p(M))= \Xi(\mathcal{A})
$
by definition of $\Xi$.
\end{proof}

Singular values are harmonic, so every class in $H_{\text{\upshape{inf}}}^p(X)$ has a harmonic representative.
As noted in section \ref{kohom}, every class in $H^p_!(X)$ has a unique $L^2-$harmonic representative, because $X$ is a complete Riemannian manifold. 
This leads to
\begin{cor}\label{co1}
Every class in  $H_{\text{\upshape{inf}}}^p(X)$ has a  representative in the singular values.
Every class in $H^p(X)$ has a harmonic representative.
\end{cor}

Unlike classes in $H_!^p(X)$, in general the classes in $H_{\text{inf}}^p(X)$ do not belong to $L^2\Omega^p(X)$, unless they are represented by residues.
Because of Lemma \ref{randrham}, closed cusp forms are in $\ker r$. 
Therefore we have a further decomposition
\[
 H_!^p(X)=H_0^p(X)\oplus H_{\text{Eis}}^p(X)
\]
Here by definition harmonic representatives of $H_0^p(X)$ are rapidly decreasing on $Z$ and thus orthogonal to fiber harmonic forms. In \cite[\S 2]{harder2}  it is shown that $H_{\text{Eis}}^p(X)$ generally is not empty.

Finally we emphasize that the decomposition
$
 H^p(X)=H_!^p(X)\oplus H_{\mathrm{inf}}^p(X)
$
does not give rise to a corresponding decomposition of  the $L^2-$cohomology  $H_{(2)}(X)$, since there are classes in $H^p_{(2)}(X)$, which are not in $H^p(X)$.
\subsection{A Signature Formula}
Theorem \ref{thm1} can be applied to compute the $L^2-$signature of $X$ as follows. Let $N=\dim X=4 l$.
Then $L^2-\sign(X)$ is the signature defined by  $H^{2l}_{(2),\text{red}}(X)\cong \H_{(2)}^{2l}(X)$, i.e. it equals the signature of the quadratic form
\[
 Q(\phi,\psi)=\int_X \phi\wedge\psi,\quad \phi,\psi\in \H_{(2)}^{2 l}(X).
\]
Similarily let $\sign_c(X)=\sign(X_0,\partial X_0)$ be the signature of $Q$ on $H_!(X)$.
Since every class in $H_!(X)$ has a unique $L^2-$harmonic representative, we get an isomorphism between $H_!(X)$  and the space $\H_!(X)$  spanned by all $L^2-$harmonic representatives of $H_!(X)$. This shows  that $\sign(X_0,\partial X_0)$ equals the signature of $Q$ on $\H_!^{2l}(X)$.

Now lets recall how  $L^2-\sign(X)$ can be computed.
The operator
\[
 \tau_X\defgleich i^{p(p-1)+2 l}\ast : \H_{(2)}^p(X)\to \H_{(2)}^{N-p}(X)
\]
is an involution on $\H_{(2)}^{2l}(X)$.
Let $\H_{\pm}^{2l}(X)$ be the $\pm 1$-Eigenspaces of $\tau_X=\ast$ on $\H_{(2)}^{2l}(X)$. 
Then
\[
 L^2-\sign(X)=\dim \H_{+}^{2l}(X)-\dim \H_{-}^{2l}(X).
\]

\begin{satz}\label{gleichesign}
Under the conditions (A) and (B)  for the manifold $X$ with fibered cusp metric the equality
\[
L^2-\sign(X)= \sign(X_0,\partial X_0)
\]
holds.
\end{satz}

\begin{bem}
The results of \cite{dai} and \cite{vaill} show
\[
 L^2-\sign(X)=\sign(X_0,\partial X_0)+\tau
\]
where $\tau$ is a topological invariant of the spectral sequence of $M\to B$. Under the conditions (A) and (B) one can conclude as in Dai's paper that $\tau=0$. 
Instead of using this result, in the following proof the difference of the signatures under consideration will be calculated directly. In this way also additional information about singular values is obtained.
\end{bem}

\begin{proof}
Let
\[
 \hat\H^p(Z)=\H^p(M)\oplus du\wedge \astm \H^p(M).
\]
Define $\tau_Z \psi\defgleich i^{p(p-1)+2 l}\astz \psi$ for the  Hodge-Star-Operator $\astz$ on $Z=\Reell^+\times M$ equipped with the metric $du^2+g^M$.
For each $\omega\in\hat\H^p(Z)$ an $L^2-$harmonic form is defined by
\[
 \tE(\omega) =\text{res}_{s=2\ud(\omega)} E(s,\omega)\defgleich\sum_k\text{res}_{s=2 d_k} E(s,\omega^{[k]}).
\]

\begin{lem}
$\tau$ commutes with the construction of $\tE$,
\begin{equation}
 \tau_X \tE(\omega) = \tE(\tau_Z\omega)\label{tauswap} 
\end{equation}
\end{lem}
\begin{proof}
Let $\phi\in \H^{p-k,k}(M)$ with $k<f/2$, $\phi\in \bild \tC^{[k]}$. Equations \eqref{duformel} and \eqref{deformel} show
\[
 dE(s,\astm\phi)=(s-2d_k)E(s,du\wedge\astm\phi)=(s-2d_k)\ast^2(\astm)^2\ast E(s,\phi)
\]
For $s\to 2d_k$:
\[
 dE(2d_k,\astm\phi)=\tE(du\wedge\astm\phi)=(-1)^p \ast\tE(\phi),
\]
and both $\tE(\phi)$ and are $\ast\tE(\phi)$ square integrable.
But now
\[
 \astz(du\wedge \phi)=\astm \phi, \quad \astz\phi=(-1)^p du\wedge \astm\phi,
\]
so that 
\[
 \ast\tE(\phi)=\tE(\astz\phi).
\]
Finally $\tE(\phi)=0$ for $\phi \in \H^{p-k,k}(M)$ with $k\ge f/2$ or $\phi\in \ker\tC^{[k]}$. 
\end{proof}


Let $h=2l$ and $\mathfrak{I}^k(M)\defgleich\bild \tC^{[k]}\subset \H^h(M), k<f/2$.
\[
\mathfrak{I}^k(Z)\defgleich  \mathfrak{I}^k(M)+du\wedge \astm \mathfrak{I}^k(M)\subset \hat\H^h(Z),\qquad 0\le k<f/2
\]
is invariant under $\tau_Z$. 
Let $Q_{k,\pm}\subset \mathfrak{I}^k(Z)$ be the  $\pm 1-$eigenspaces of $\tau_Z$ and $Q_{\pm}=\bigoplus_{k<f/2} Q_{k,\pm}$.

Let $W\defgleich\tE(\hat\H^h(Z))\subset \H_{(2)}^h(X)$ and
$W_\pm\defgleich\tE(Q_\pm)$.
Because of \eqref{tauswap} we have $W_\pm\subset \H_{\pm}^h(X)$.

\begin{lem}\label{gleichdim}
\[
 W=W_+\oplus_Q W_-\qquad\text{and}\qquad \dim W_+=\dim W_-
\]
\end{lem}
\begin{proof}
All  differential forms under consideration have even total degree
$h$, so that $\tau=\ast$. First we show
 \[
 Q_{k\pm}=\{\phi\pm\tau_Z\phi\mid \phi\in \mathfrak{I}^k(M)\}, \quad 0\le k<f/2.
\]
Obviously the right-hand side is contained  in $Q_{k\pm}$. Let $\psi\in Q_{k\pm}$ and
write $\psi=\phi+du\wedge\astm \phi_1$ with $\phi, \phi_1\in \mathfrak{I}^k(X)$. 
Then 
\[
\tau_Z \psi=du\wedge \astm \phi+\phi_1 =\pm\psi.
\]
Contraction with $\frac{\partial}{\partial u}$ shows $\phi_1=\pm \phi$, which implies
$\dim Q_{k+}=\dim Q_{k-}$.

The further statements now follow from
 $\tE(\hat\H^h(Z))=\tE(Q_+\oplus Q_-)$ (by construction), and from $\tE(\eta)\neq 0$   for  $\eta\in Q_+\oplus Q_-$.
\end{proof}


%

Let $\psi\in\H_{(2)}^h(X)$ with $r[\psi]\neq 0$. We claim that then already 
$
r[\psi]\in\bigoplus_{k<f/2}H^{*,k}.
$
 Because $\psi$ is square integrable,
$e^{-\ua u}\psi|_Z$ must lie in $L^2\Omega(Z, du^2+g^M)$. For fiber-degrees $k\ge f/2$ by definition $\ua\leq 0$, so that
$\|\psi^{[k]}|_{\{u\}\times M}\|\xrightarrow[u\to 0]{} 0$. Because in addition $d\psi=0$, we conclude $r[\psi^{[k]}]=0$ as in Lemma \ref{randrham}.

Now Theorem \ref{thm1} proves the existence of a unique
\[
\phi_0\in \bild  \sum_{0\le j<f/2} \mathfrak{I}^j(M)\subset \H^h(M)
\quad\text{ so that }\quad
r[\tE(\phi_0)]=r[\psi].
\]
If $\psi\in\H_{(2)}^h(X)$ with $r[\psi]= 0$, let $\phi_0=0$.

Let $\psi_!=\psi-\tE(\phi_0)$. Since both $\psi$ and $\tE(\phi_0)$ are $L^2-$harmonic, the same must be true for $\psi_!$. In addition
$
r[\psi_!]=0,
$
so that $\psi_!\in\H_!(X)$.
We have
\[
\psi
=\frac{1}{2}\tE(\phi_0+\tau_Z\phi_0)
+\frac{1}{2}\tE(\phi_0-\tau_Z\phi_0)
   +\psi_! .
\]
In this way we have defined an isomorphism of vector spaces
\begin{align*}
 \rho:\H_{(2)}^h(X)&\to (W_+\oplus_Q W_-)\oplus \H_!(X)\\
\psi&\mapsto (\tE(\phi_0+\tau_Z\phi_0),\;\tE(\phi_0-\tau_Z\phi_0),\;\:\psi_!).
\end{align*}
To proof injectivity of $\rho$, let $\psi\in\H_{(2)}^h(X)$ with $\rho(\psi)=0$. From $\tE(\phi_0\pm\tau_Z\phi_0)=0$ we conclude $\phi_0=0$ and $\psi=\psi_!=0$. 

To see surjectivity, let $(w_+,w_-,\psi_!)\in (W_+\oplus_Q W_-)\oplus \H_!$.  Lemma \ref{gleichdim} shows the existence of $\phi_1, \phi_2 \in  \bigoplus_{0\le k<f/2}\mathfrak{I}^k(M)$ with
\[
 w_+=\tE(\phi_1+\tau_Z\phi_1),\quad w_-= \tE(\phi_2-\tau_Z\phi_2),
\]
and
$
 \psi\defgleich\tE(\phi_1+\tau_Z\phi_2)+\psi_!
$
is mapped to $\rho(\psi)=(w_+,w_-,\psi_!)$.

Because $W=W_+\oplus_Q W_-$ and $\H_{(2)}^h(X)$ are $\ast-$invariant and $\H_{(2)}^h(X)=W\oplus \H_!(X)$, also 
$\H_!(X)$ must be $\ast-$invariant. Finally
\begin{eqnarray*}
 L^2-\sign(X)&=&\dim \H_{+}^{h}(X)-\dim \H_{-}^{h}(X) \\
&=&\dim ( \H_!^h(X)\cap \H_{+}^{h}(X))-\dim ( \H_!^h(X)\cap \H_{-}^{h}(X)).
\end{eqnarray*}
This proves Proposition \ref{gleichesign}.
\end{proof}

\end{document}